\documentclass[a4paper,12pt,reqno]{amsart}
\usepackage{amsthm, mathrsfs,amssymb,amsmath,}
\usepackage{enumerate}
\usepackage{tikz}
\usepackage{pgfplots}
\usepackage{mathtools}
\usepackage[foot]{amsaddr}
\makeatletter
\@namedef{subjclassname@2010}{%
	\textup{2010} Mathematics Subject Classification}
\makeatother

\pgfplotsset{compat=1.18}

\usepackage{hyperref}
\allowdisplaybreaks
\numberwithin{equation}{section}

\usepackage{graphicx}
\usepackage{enumitem}

\usepackage{graphicx} 
\newtheorem{theorem}{Theorem}[section]
\newtheorem{lemma}[theorem]{Lemma}
\usepackage{xcolor}
\usepackage{comment}

\theoremstyle{definition}
\newtheorem{definition}[theorem]{Definition}

\newtheorem{proposition}[theorem]{Proposition}
\newtheorem{corollary}[theorem]{Corollary}
\newtheorem{note}[theorem]{Note}

\theoremstyle{remark}
\newtheorem{remark}[theorem]{Remark}

\numberwithin{equation}{section}

\definecolor{trp}{rgb}{1,1,1}

\definecolor{red}{rgb}{1,0,.2}

\definecolor{blue}{rgb}{0,0,1}

\definecolor{rgrey}{rgb}{.8,0.4,.4}  
\definecolor{grey}{rgb}{.13,.13,.13}  

\definecolor{green}{rgb}{0.0,0.4,0.2}

\numberwithin{equation}{section}

\newcommand{\R}{\mathbb{R}}
\newcommand{\N}{\mathbb{N}}

\newcommand{\0}{\underline{0}}

\allowdisplaybreaks

\begin{document}
	
	\title[]{Hausdorff dimension of the Wedding cake type surfaces}
	
	
	\author{Bal\'azs B\'ar\'any}
	\author{Manuj Verma}
        \thanks{BB and MV acknowledge support from the grants NKFI~K142169 and NKFI KKP144059 ``Fractal geometry and applications" Research Group.}
	\address{Department of Stochastics \\ Institute of Mathematics \\  Budapest University of Technology and Economics \\ M\H{u}egyetem rkp. 3., H-1111 Budapest, Hungary }
	\address[Bal\'azs B\'ar\'any]{barany.balazs@ttk.bme.hu}
	\address[Manuj Verma]{mathmanuj@gmail.com }

	
	%

	
	\date{\today}
	
	
	\begin{abstract} In this paper, we study the Hausdorff dimension of fractal interpolation surfaces (FISs) over a triangular domain.
    These FISs are known as `wedding cake surfaces'. These surfaces are the attractor of some deterministic self-affine iterated function systems (IFS) on $\mathbb{R}^3$ generated by a fractal interpolation algorithm. Due to the recent seminal result of Rapaport (Adv. Math. 449 (2024) 109734), the dimension theory of self-affine IFS on $\mathbb{R}^3$ is known whenever the IFS is strongly irreducible and proximal. However, the self-affine IFSs associated with FIS are not strongly irreducible. We prove that the Hausdorff dimension of the self-affine set (or FIS) is the same as the affinity dimension outside a set of scaling parameters with zero Lebesgue measure. Lastly, by computing the overlapping number for the associated Furstenberg IFS, we determine the Hausdorff dimension for every type of scaling parameter in a certain range of parameters.  
		
	\end{abstract}
	\maketitle

	
	\section{Introduction and Statements}

\subsection{Historical background.} For $n\geq 2$, the system $\mathcal{I}=\big\{ f_1,f_2,\ldots,f_n \big\}$ is called an \texttt{iterated function system (IFS)} on $\mathbb{R}^d$, if the map $f_i$ is a contraction map on $\mathbb{R}^d$ for each $i\in \{1,2,\ldots, n\}$. Hutchinson \cite{Hutchinson} proved that there exists a unique non-empty compact set $K\subset \mathbb{R}^d$ such that 
$K=\bigcup_{i=1}^{n}f_i(K).$ 
Let $\mathbf{p}=(p_1,p_2,\dots,p_n)$ be a probability vector. Hutchinson \cite{Hutchinson} also proved that there exists a unique Borel probability measure $\mu$ supported on $K$ such that 
$\mu(B)=\sum_{i=1}^{n}p_{i}\mu(f_{i}^{-1}(B))$
for all Borel sets $B\subset \mathbb{R}^d$. The set $K$ is called the \texttt{attractor} of IFS $\mathcal{I}$ and the measure $\mu$ is known as the \texttt{stationary measure} corresponding to IFS $\mathcal{I}$ with the probability vector $\mathbf{p}$. The map $f: \mathbb{R}^d \to \mathbb{R}^d$ such that $f(x)=Ax+a$ is called a self-affine map, where $a\in \mathbb{R}^d$ and $A\in \text{GL}({d},\mathbb{R})$ with $\|A\|<1.$ The IFS $\mathcal{I}=\big\{ f_1,f_2,\ldots,f_n \big\}$ is called a self-affine IFS, if each $f_i$ is a self-affine map on $\mathbb{R}^d$ for each $i\in \{1,2,\dots, n\}$. The attractor of the self-affine IFS is known as \texttt{self-affine set}, and the stationary measure corresponding to the self-affine IFS $\mathcal{I}$ and probability vector $\mathbf{p}$ is known as \texttt{self-affine measure}. \par 
In this paper, our focus is on the Hausdorff dimension of the self-affine IFS on $\mathbb{R}^3$, which are generated by the fractal interpolation algorithm on the triangular domain of $\mathbb{R}^2$. In 1986, Barnsley \cite{MF1} introduced the concept of fractal interpolation functions (FIFs) on $\mathbb{R}$. The FIF is a function which interpolates the given data set, and the graph of this function is the attractor of some iterated function system (IFS). In \cite{MF1}, Barnsley provided an algorithm to construct an IFS corresponding to the given data set. Barnsley, Elton and Hardin \cite{MF2} determined the box dimension of the graph of the FIF for a given data set. Later, B\'ar\'any, Simon and Rams \cite{BRS2020} determined the Hausdorff
dimension of the graph of the FIF by studying the dimension theory of the associated self-affine IFS on $\mathbb{R}^2$ for a given data set on $\mathbb{R}$. \par 
In 1990, Massopust \cite{Mass1990} extended Barnsley's FIF theory on the plane and defined the notion of the fractal interpolation surfaces (FISs) and also provided an algorithm to construct a self-affine IFS on $\mathbb{R}^3$ corresponding to a given finite data sets over the triangular domain, where the data points on the boundary of the triangular domain are required to be coplanar. Barnsley called these surfaces ``wedding cake" surfaces. Under the consideration of uniform triangulation of the equilateral triangle and the linear part of the self-affine IFS does not contain the rotation matrix, Massopust \cite{Mass1990} determined the box dimension of the graph of corresponding FISs. Geronimo and Hardin \cite{GH1993} provided another construction of the FISs over the triangular domain and polygonal domain by considering uniform scaling parameters without assuming the coplanarity condition as in \cite{Mass1990}, and also determined the box dimension of the FISs under some condition. We note that the construction of the FISs on the rectangular domain was found in \cite{Dalla, Ruan}. \par 
According to our knowledge, the Hausdorff dimension of the FISs has not yet been studied in these cases.  The dimension theory of the FISs is equivalent to the dimension theory of the attractor of the corresponding self-affine IFS on $\mathbb{R}^3$. In 1988, Falconer \cite{Falconer1988} introduced a natural upper bound for the Hausdorff dimension of the self-affine sets in $\mathbb{R}^d$, which is known as the affinity dimension. Falconer \cite{Falconer1988} proved that if the self-affine IFS $\mathcal{I}=\{f_i(x)=A_ix+a_i\}_{i=1}^{n}$ with attractor $K$ satisfies $\|A_i\|<\frac{1}{3}~\forall~i\in\{1,2,\dots,n\}$, then 
for almost all ${\bold{a}}=(a_1,\dots,a_n)\in \mathbb{R}^{nd}$ 
\begin{equation}\label{Falconer}
\dim_{H}(K)=\min\{d,t\},
\end{equation}
where $t$ is the affinity dimension of the self-affine IFS $\mathcal{I}$ (see precise definition later in Section~\ref{sec:prelim}). After that, Solomyak \cite{Boris1998} showed that the Falconer's formula \eqref{Falconer} is also valid whenever $\|A_i\|<\frac{1}{2} ~\forall~i\in\{1,2,\dots,n\}$ and the bound $\frac{1}{2}$ is strict. In the planar case, B\'ar\'any, Hochman and Rapaport \cite{BHR} proved that if the self-affine IFS $\mathcal{I}$ is strongly irreducible and proximal, and satisfies the strong open set condition (SOSC), i.e. there exists an open and bounded set $U$ such that $$f_i(U)\subseteq U,\, f_i(U)\cap f_j(U)=\emptyset\text{ for $i\neq j$ and }U\cap K\neq\emptyset,$$ then \eqref{Falconer} holds. Later, Hochman and Rapaport \cite{HRESC} determined the more general result in the planar case when the maps in the self-affine $\mathcal{I}$ do not have a common fixed point, $\mathcal{I}$ is strongly irreducible and proximal, and satisfies the exponential separation condition. In the case $d=3$, Rapaport \cite{Rapa2024} proved recently that \eqref{Falconer} holds if the self-affine IFS $\mathcal{I}$ is strongly irreducible and proximal, and satisfies the SOSC.\par
Note that the self-affine IFS on $\mathbb{R}^3$ generated by the fractal interpolation algorithm for a given data set is not strongly irreducible; it is actually reducible. Thus, \cite{Rapa2024} is not applicable for studying the Hausdorff dimension of the FISs on $\mathbb{R}^3$.\par 
\subsection{ Massopust's Fractal Interpolation Surfaces.}\label{sec:massopust} First, we take the construction of the FISs given by Massopust \cite{Mass1990}. For determining the Hausdorff dimension, we consider the same assumption as taken by Massopust \cite{Mass1990} for the box dimension. We consider the equilateral triangle $\Delta$ with vertices $\{(0,0),(1,0),(\frac{1}{2},\frac{\sqrt{3}}{2})\}$ and an integer $N\geq3$. Then we divide each side into $N$ equal parts, we get a uniform triangulation $\{\Delta_i\}_{i=1}^{N^2}$ (see Figure~\ref{fig:triang} for $N=3$). Without loss of generality, we index the triangles $\Delta_i$, which are pointing up, by indices $i=1,\ldots,\frac{N(N+1)}{2}$, and the triangles which pointing down by $i=\frac{N(N+1)}{2}+1,\ldots,N^2$.

Let us denote the vertices of this triangularization on the plane by $\{q_i\}_{i=1}^{L(N)}$, where $L(N)=\frac{(N+1)(N+2)}{2}$ for $N\geq 3$. We will use the convention that $q_1,q_2$ and $q_3$ denote the vertices of the original equilateral triangle counted from the bottom left corner in anti-clockwise direction. We consider a data set $\{(q_k,a_k)\}_{k=1}^{L(N)}$ associated with the triangulation $\{\Delta_i\}_{i=1}^{N^2}$. We assume that the data points on the boundary of the equilateral triangle $\Delta$ are coplanar. Without loss of generality, we assume that $a_k=0$ for all $k$ such that the corresponding $q_{k}$ is on the boundary of the triangle $\Delta$. In particular, the data set is $\{(q_{i},0)\}_{i=1}^{9}\cup \{(q_{10}, a)\},$ for the case $N=3$, where $a\ne 0$ is a real number. We define the map $f^*\colon\{q_i\}_{i=1}^{L(N)}\to\{a_i\}_{i=1}^{L(N)}$ as follows $f^{*}(q_k)=a_k$ for $k=1,\ldots,L(N)$.

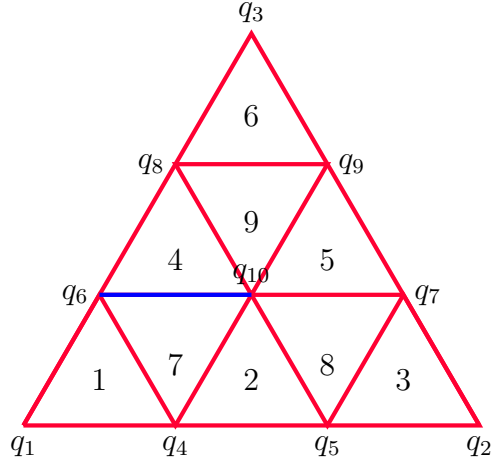
\begin{figure}
	\begin{center}
		\begin{tikzpicture}
            \draw[ultra thick, red] (-3,0)--(0, sqrt 27)--(3,0)--(-3,0);
			\node[below] at (-3,0) {$q_1$};
			\node[below] at (3,0) {$q_2$};
			\node[above] at (0,sqrt 27) {$q_{3}$};
			\draw[ultra thick, red](-3,0)--(-2, {( sqrt (27)/3)})--(-1,0)--(0,{( sqrt (27)/3)})--(1,0)--(2, {( sqrt (27)/3)})--(3,0);
          \draw[ultra thick, red](-2, {( sqrt (27)/3)})--(0,{( sqrt (27)/3)})--(-1,{( sqrt (108)/3)})--(1,{( sqrt (108)/3)})--(0,{( sqrt (27)/3)})--(2, {( sqrt (27)/3)});  
			\node[below] at (-1,0) {$q_4$};
			\node[left] at (-2, {( sqrt (27)/3)}) {$q_6$};
            \node[left] at (-1,{( sqrt (108)/3)}) {$q_8$};
			\node[right] at (1,{( sqrt (108)/3)}) {$q_9$};
            \node[below] at (1,0) {$q_5$};
        \draw[ultra thick, blue](-2, {( sqrt (27)/3)})--(0,{( sqrt (27)/3)});    
			\node[right] at (2, {( sqrt (27)/3)}) {$q_7$};
            \node[above] at (0,{( sqrt (27)/3)}) {$q_{10}$};
		\node at (-2,0.6) {1};
			\node at (0,0.6) {2};
            \node at (2,0.6) {3};
    \node at (-1,0.8) {7};
    \node at (1,0.8) {8};
    \node at (-1,2.2) {4};
    \node at (1,2.2) {5};
     \node at (0,2.7) {9};
      \node at (0,4.1) {6};
		\end{tikzpicture}  
	\end{center}
    \caption{Triangularization of the equilateral triangle for $N=3$. To visualise the sign of the change of the data set on each horizontal edge, we coloured it blue where the data set decreases from left to right, and otherwise, we coloured it red.}\label{fig:triang}
\end{figure}


For each $i\in \{1,2,\dots, N^2\}$, we denote the value of $f^*$ at the left vertex of the horizontal edge of $\Delta_i$ by $a_{1}^{i}$, value of $f^*$ at the right vertex of the horizontal edge by $a_{2}^{i}$ and value at the other vertex by $a_{3}^{i}$ of $\Delta_i$, where $a_{1}^{i},a_{2}^{i},a_{3}^{i}\in \{a_k\}_{k=1}^{L(N)}.$ For each $i\in \{1,2,\dots,N^2\}$, we define a similarity map $U_{i}: \Delta\to \Delta_i$ such that 
 \begin{equation}\label{eq:Ui}U_i(x,y)=\begin{cases}
     \big(\frac{x}{N}, \frac{y}{N}\big)+(e_i,f_i) &\text{if}~ a_{1}^{i}\geq a_{2}^{i}\text{ and }1\leq i\leq\frac{N(N+1)}{2},\\
     \big(\frac{x}{N}, -\frac{y}{N}\big)+(e_i,f_i) &\text{if}~ a_{1}^{i}\geq a_{2}^{i}\text{ and }\frac{N(N+1)}{2}+1\leq i\leq N^2,\\
     \big(-\frac{x}{N},\frac{y}{N}\big)+(g_i,f_i) &\text{if}~ a_{1}^{i}< a_{2}^{i}\text{ and }1\leq i\leq\frac{N(N+1)}{2},\\
     \big(-\frac{x}{N},-\frac{y}{N}\big)+(g_i,f_i) &\text{if}~ a_{1}^{i}< a_{2}^{i}\text{ and }\frac{N(N+1)}{2}+1\leq i\leq N^2,\\
 \end{cases}\end{equation}
 where $(e_i,f_i)$ and $(g_i,f_i)$ are the left and right vertices of the horizontal line of $\Delta_i$. Furthermore, we define for each $i\in \{1,2,\dots,N^2\}$, the height function as
	\begin{equation}\label{eq:Vi}
    V_i(x,y,z)=\bold{a}_i x+ b_i y+ s_i z+ c_i,
    \end{equation}
    where $s_i\in (0,1)$ and constants $\{\bold{a}_i, b_i, c_i\}$ are uniquely determined by the join-up condition i.e. 
    \begin{equation}\label{eq:boundcond}
    V_i(q_1,0)=f^*(U_i(q_1)),V_i(q_2,0)=f^*(U_i(q_2))\text{ and }V_i(q_3,0)=f^*(U_i(q_3))
    \end{equation}
    for every $i\in \{1,2,\dots,N^2\}$.

We define the affine IFS $\mathcal{I}:=\{W_i,i\in \{1,2,\dots,N^2\}\}$ on $\mathbb{R}^3$ such that the maps $W_i: \mathbb{R}^3 \to \mathbb{R}^3$ are defined as 
\begin{equation}\label{eq:Wi}
W_i(x,y,z)=(U_i(x,y), V_i(x,y,z)).
\end{equation}

By \cite{Mass1990}, the map $f^*$ defined only on the date set can be uniquely extended to a continuous function $f^*: \Delta \to \mathbb{R}$ such that $f^*(q_i)=a_i$ for all $i\in\{1,2\dots, L(N)\}$ and the graph $G({f^*})$ of the function $f^{*}$ is the attractor of the affine IFS $\mathcal{I}.$ In particular, it satisfies the equation:
    $$
    f^*(U_i(x,y))=V_i(x,y,f^*(x,y))\text{ for every }i\in\{1,\ldots,L(N)\}.
    $$
    The surfaces $G({f^*})$ are known as fractal interpolation surfaces. The maps of the IFS $\mathcal{I}$ in the case $N=3$ are precisely as follows:
	\begin{align*}
		W_1(x,y,z)&=\bigg(\frac{x}{3},\frac{y}{3},s_1z\bigg),\quad W_4(x,y,z)=\bigg(\frac{-x}{3}+\frac{1}{2},\frac{y}{3}+\frac{1}{2\sqrt{3}},-ax-\frac{a}{\sqrt{3}}y+s_4z+a\bigg)\\
		W_2(x,y,z)&=\bigg(\frac{x+1}{3},\frac{y}{3},\frac{2a}{\sqrt{3}}y+s_2z\bigg),~W_{9}(x,y,z)=\bigg(\frac{x}{3}+\frac{1}{3},\frac{-y}{3}+\frac{1}{\sqrt{3}},\frac{2a}{\sqrt{3}}y+s_9z\bigg)  \\
W_6(x,y,z)&=\bigg(\frac{x}{3}+\frac{1}{3},\frac{y}{3}+\frac{1}{\sqrt{3}},s_6z\bigg)~W_{8}(x,y,z)=\bigg(\frac{x}{3}+\frac{1}{2},\frac{-y}{3}+\frac{1}{2\sqrt{3}},-ax-\frac{a}{\sqrt{3}}y+s_8z+a\bigg),\\
        W_3(x,y,z)&=\bigg(\frac{x+2}{3},\frac{y}{3},s_3z\bigg),~W_7(x,y,z)=\bigg(\frac{-x}{3}+\frac{1}{2},\frac{-y}{3}+\frac{1}{2\sqrt{3}},-ax-\frac{a}{\sqrt{3}}y+s_7z+a\bigg)\\W_{5}(x,y,z)&=\bigg(\frac{x}{3}+\frac{1}{2},\frac{y}{3}+\frac{1}{2\sqrt{3}},-ax-\frac{a}{\sqrt{3}}y+s_5z+a\bigg).
	\end{align*}
\begin{figure}[h!]
	\begin{minipage}{0.6\textwidth}
\includegraphics[width=1.0\linewidth]{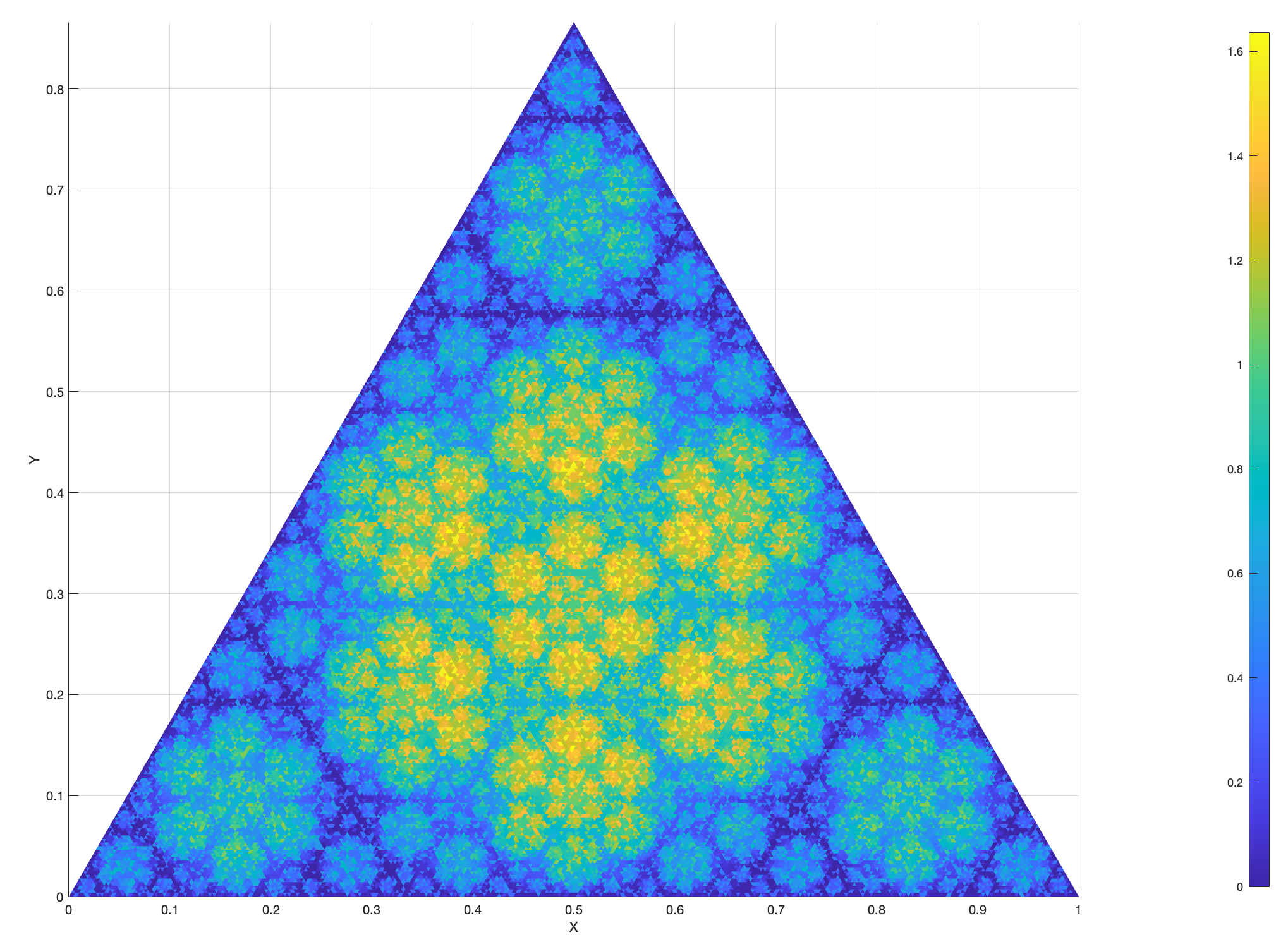}
	\end{minipage}\hspace*{0.2mm}
	\begin{minipage}{0.4\textwidth}\includegraphics[width=1.0\linewidth]{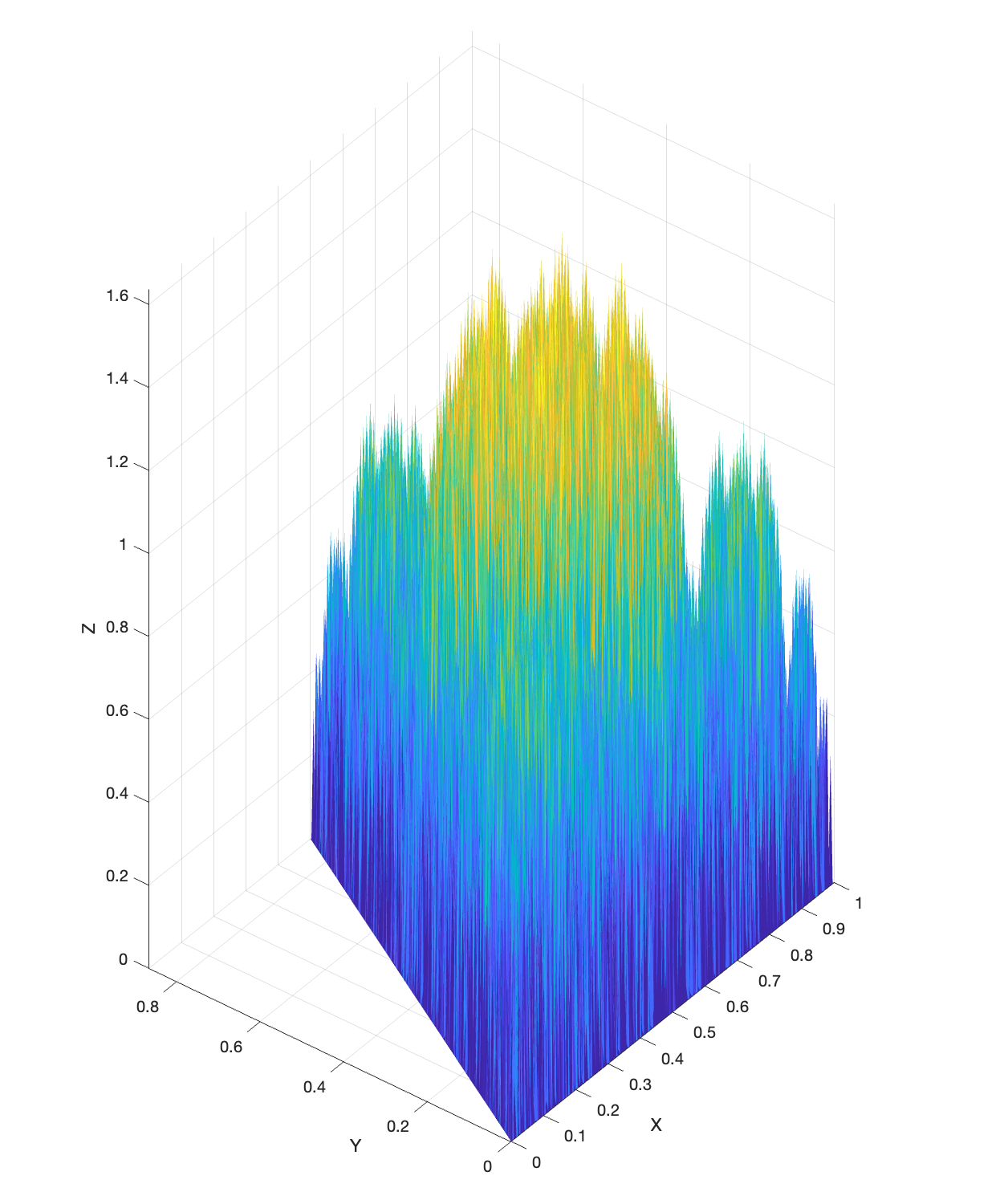}
	\end{minipage}	
    \caption{Graph of the fractal interpolation surface (top \& side view) with parameters $N=3$, $s_i=s=0.75$ and $a=1$.}
\end{figure}

 Define 
$$\mathcal{A}_1:=\{i\in\{1,2,\dots,N^2\}: a_{1}^{i}=a_{2}^{i}\},\quad  \mathcal{A}_2:=\{i\in\{1,2,\dots,N^2\}: a_{1}^{i}>a_{2}^{i}\}$$
$$\mathcal{A}_3:=\{i\in\{1,2,\dots,N^2\}: a_{1}^{i}<a_{2}^{i}\}.$$
Moreover, set 
$$D=\min_{k,i\in \mathcal{A}_3}\{\frac{|a_{1}^{k}-  a_{2}^{k}|}{|a_{1}^{i}-  a_{2}^{i}|}\}\text{ and }B=\frac{1}{1+\max_{i\in \mathcal{A}_2}\{B_i\}},\text{ where }B_{i}=\max_{k\in\mathcal{A}_{3}}\{\frac{|a_{1}^{i}-  a_{2}^{i}|}{|a_{1}^{k}-  a_{2}^{k}|}\}$$
for all $i\in \mathcal{A}_2$. Our result on the dimension of the graph $f^*$ is as follows:  
\begin{theorem}\label{Dimalmost}
    Let $\mathcal{I}:=\{W_i,i\in \{1,2,\dots,N^2\}\}$ be a self-affine IFS on $\mathbb{R}^3$ defined as above.  Let $G(f^*)$ be the attractor of the IFS $\mathcal{I}$. For each $i\in \{1,2,\dots,N^2\}$, we consider  $a_{1}^{i}\ne a_{2}^{i}$
if $a_{1}^{i}$ and $a_{2}^{i}$ both are not  on the boundary of original triangle $\Delta$. Then,
    $$\dim_{H}(G(f^*))=\dim_{B}(G(f^*))=1+\frac{\log(\sum_{i=1}^{N^2}s_i)}{\log N} $$ 
    for Lebesgue almost every scaling parameter $\underline{s}\in (\frac{1}{N},1)^{\#\mathcal{A}_1}\times (\frac{1}{N B},1)^{\#\mathcal{A}_2}\times (\frac{1}{N D},1)^{\#\mathcal{A}_3}$. 
\end{theorem}
In the case of $N=3$, one can see that $\#\mathcal{A}_1=5,\#\mathcal{A}_2=2, \#\mathcal{A}_3=2$ and $B=\frac{1}{2}$ and $D=1$. Thus, we have the following Corollary for the typical type results for the FISs.
    \begin{corollary}
       For Lebesgue almost every $\underline{s}=(s_1,s_2,s_3,s_4,s_5,s_6,s_7,s_8,s_9)$
       such that $s_1,s_2,s_3,s_{4},s_6,s_7,s_9\in (\frac{1}{3},1)$ and $s_5,s_8\in (\frac{2}{3},1)$, we have $$\dim_{H}(G(f^*))=\dim_{B}(G(f^*))=1+\frac{\log(\sum_{i=1}^{9}s_i)}{\log(3)}.$$  
    \end{corollary}

 Our second main result gives the Hausdorff dimension of the FIS in the case of $N=3$ for every parameter in a certain region.
 \begin{theorem}\label{main2}
		Let $s_i\in (\frac{2}{3},1)$ for every $i\in \{1,2,\dots,9\}$. If $\max\{s_{5},s_{8}\}\leq \min \{s_{4},s_{7}\}$ and $s_2\leq s_9$, then $$\dim_{H}(G(f^*))=\dim_{B}(G(f^*))=1+\frac{\log(\sum_{i=1}^{9}s_i)}{\log(3)}.$$
	\end{theorem}
    
    \subsection{Geronimo-Hardin FISs} Next, we consider the construction of the FISs given by Geronimo and Hardin \cite{GH1993}. In this construction, the data points on the boundary of the triangle $\Delta$ do not need to be coplanar but need to take uniform scaling parameters. Consider the equilateral triangle $\Delta$ with vertices $\{q_{1}=(0,0), q_{2}=(1,0), q_{3}=(\frac{1}{2},\frac{\sqrt{3}}{2})\}.$ We take the triangulation $\{\Delta_i\}_{i=1}^{4}$ as shown Figure~\ref{fig:GH}.

    \begin{figure}[h!]
	\begin{center}
		\begin{tikzpicture}
			\draw[ultra thick, red](-2,0)--(2,0)--(0,sqrt 12)--(-2,0);
			\node[below] at (-2,0) {$q_1$};
			\node[below] at (2,0) {$q_2$};
			\node[above] at (0,sqrt 12) {$q_3$};
			\draw[ultra thick, red](0,0)--(-1, sqrt 3)--(1,  sqrt 3)--(0,0);
			\node[below] at (0,0) {$q_4$};
			\node[left] at (-1, sqrt 3) {$q_5$};
			\node at (-1,0.5) {1};
			\node at (0,1.2) {4};
			\node at (1,0.5) {2};
			\node at (0,2.5) {3};
		\node[right] at (1, sqrt 3) {$q_6$};
			\node[above left] at (-2,0) {$\color{blue}(1)$};
			\node[above right] at (2,0) {$\color{green}(2)$};
			\node[above] at (0,0) {$\color{violet}(3)$};
			\node[above left] at (-1, sqrt 3) {$\color{green}(2)$};
			\node[above right] at (1, sqrt 3) {$\color{blue}(1)$};		
			\node[left] at (0, sqrt 12) {$\color{violet}(3)$};
		\end{tikzpicture}  
	\end{center}
    \caption{Triangularization in the Geronimo-Hardin construction.}\label{fig:GH}
    \end{figure}
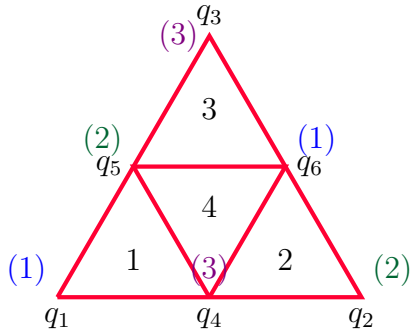
	We consider a data set as follows $$\{(q_{1},0),(q_{2},0),(q_{3},0),(q_{4},a),(q_{5},a),(q_{6},a)\},$$ where $a\ne 0$ is a real number, and the basic data function $f^*(q_i)=0$ if $i=1,2,3$ and $f^*(q_i)=a$ if $i=4,5,6$. The chromatic number for this graph is $3$.  We can label the vertices with three colours (blue, green, violet). For each $i\in \{1,2,3,4\}$, the similarity map $U_i: \Delta \to \Delta_{i}$ is defined such that the map $U_i$ maps the vertex of $\Delta$ of a color (blue or green or violet) to the vertex of same color (blue or green or violet) of $\Delta_i$. And, for each $i\in \{1,2,3,4\}$, the height function $V_i: \Delta \times \mathbb{R}\to \mathbb{R}$ is defined by 
	$V_i(x,y,z)=\bold{a}_i x+ b_i y+ s z+ c_i$, where $s\in (0,1)$ and constants $\{\bold{a}_i, b_i, c_i\}$ are uniquely determined by the join-up condition i.e. $V_i(q_1,0)=f^*(U_i(q_1)),V_i(q_2,0)=f^*(U_i(q_2))$ and $V_i(q_3,0)=f^*(U_i(q_3))~\forall~ i\in \{1,2,\dots,4\}$. We define the self-affine IFS $\mathcal{I}:=\{W_i,i\in \{1,2,\dots,4\}\},$ where the map $W_i: \mathbb{R}^3 \to \mathbb{R}^3$ is defined as $$W_i(x,y,z)=(U_i(x,y), V_i(x,y,z)).$$ By \cite{GH1993}, there exists a unique continuous function $f^*: \Delta \to \mathbb{R}$ such that the function $f^*$ interpolates the data sets and the graph $( G({f^*}) )$ of the function $f^{*}$ is the attractor of the affine IFS $\mathcal{I}.$ Precisely, the maps in the IFS $\mathcal{I}$ are as follows:
         \begin{align*}
     &W_1(x,y,z)=\begin{bmatrix}  \frac{1}{4} &\frac{\sqrt 3}{4} & 0 \\ \frac{\sqrt 3}{4} & \frac{-1}{4} & 0\\ a & \frac{a}{\sqrt 3} & s \end{bmatrix} \begin{pmatrix} x \\ y \\ z,  \end{pmatrix},~  W_2(x,y,z)=\begin{bmatrix}  \frac{1}{4} & \frac{-\sqrt 3}{4} & 0 \\ \frac{-\sqrt 3}{4} & \frac{-1}{4} & 0\\ -a & \frac{a}{\sqrt{3}} & s \end{bmatrix} \begin{pmatrix} x \\ y \\ z  \end{pmatrix}+\begin{pmatrix} \frac{3}{4}\\ \frac{\sqrt 3}{4} \\ a  \end{pmatrix},\\&
		W_3(x,y,z)=\begin{bmatrix}  \frac{-1}{2} & 0 & 0 \\ 0 & \frac{1}{2} & 0\\ 0 & \frac{-2a}{\sqrt{3}}& s \end{bmatrix} \begin{pmatrix} x \\ y \\ z  \end{pmatrix}+\begin{pmatrix} \frac{3}{4}\\ \frac{\sqrt 3}{4} \\ a  \end{pmatrix},~ W_4(x,y,z)=\begin{bmatrix}  \frac{-1}{2} & 0 & 0 \\ 0 & \frac{-1}{2} & 0\\ 0 & 0 & s \end{bmatrix} \begin{pmatrix} x \\ y \\ z  \end{pmatrix}+\begin{pmatrix} \frac{3}{4}\\ \frac{\sqrt 3}{4} \\ a  \end{pmatrix}.   
    \end{align*}

  \begin{figure}[h!]
	\begin{minipage}{0.6\textwidth}
\includegraphics[width=1.0\linewidth]{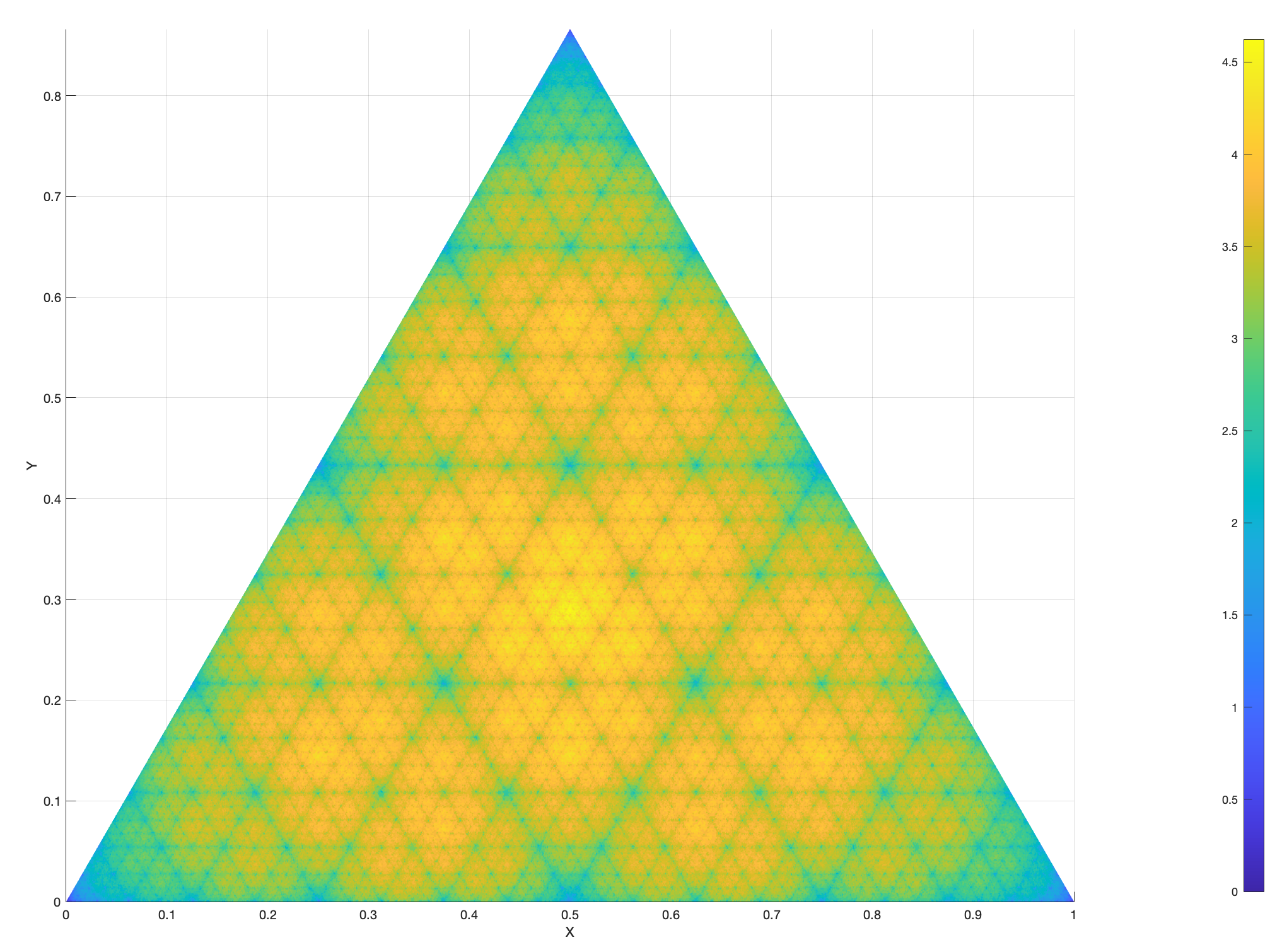}
	\end{minipage}\hspace*{0.2mm}
	\begin{minipage}{0.4\textwidth}\includegraphics[width=1.0\linewidth]{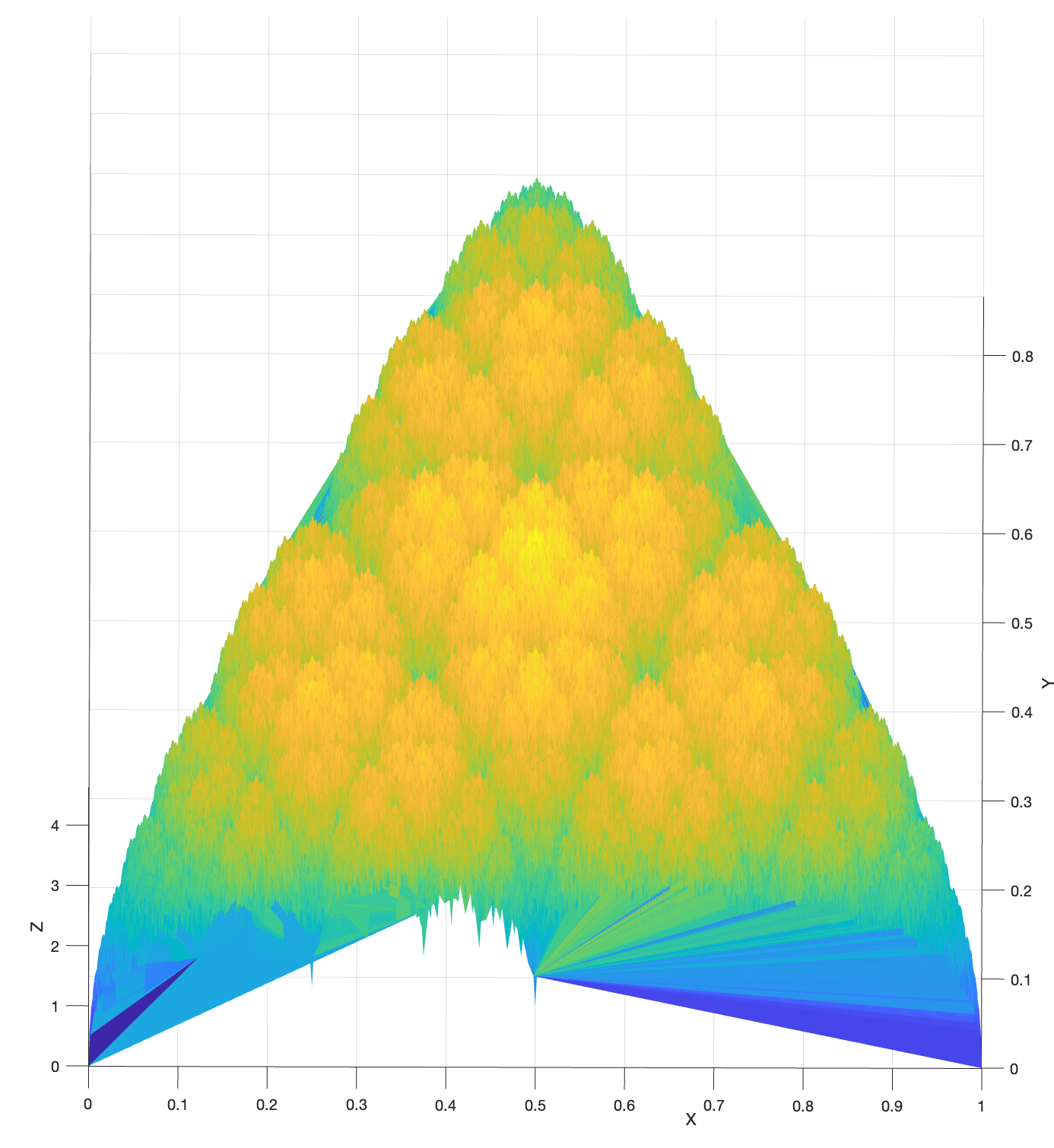}
	\end{minipage}	
    \caption{Graph of the fractal surfaces (top  \& aerial view) with parameters $s=0.82$ and $a=1$}
    \end{figure}
First, by computing the overlapping number, we determine the dimension for every type of scaling parameter as follows:
\begin{theorem}\label{GH1}
		If $s\in \bigg[\frac{1+\sqrt 5}{4},1\bigg)$, then $$\dim_{H}(G(f^*))=\dim_{B}(G(f^*))=3+\frac{\log(s)}{\log(2)}.$$
	\end{theorem}
In this case, we also determine the dimension for typical scaling parameters as follows:
\begin{theorem}\label{GH2}
   There exists a set $\mathcal{E}\subset (\frac{1}{2},1)$ with $\dim_{H}(\mathcal{E})=0$ such that 
   $$\dim_{H}(G(f^*))={\dim}_{B}(G(f^*))=3+\frac{\log(s)}{\log(2)}\quad \forall~ s\in \bigg(\frac{1}{2},1\bigg)\setminus \mathcal{E}.$$
\end{theorem}

	\section{Preliminaries}\label{sec:prelim}

    First, we go through some basic definitions and tools we intend to use.

    \subsection{Dimension concepts}
	\begin{definition}
		Let $F\subseteq \mathbb{R}^d$. We say that $\{U_{i}\}$ is a $\delta$-cover of $F$ if $F\subset \bigcup\limits_{i=1}^{\infty}U_{i}$ and $0 <|U_{i}|\leq \delta$ for each ${i}$, where $|U_i|$ denotes the diameter of the set $U_i$.
		For each $\delta>0$ and $s\geq 0$, we define
		$$H_{\delta}^{s}(F)=\inf\Big\{\sum_{i=1}^{\infty}|U_{i}|^{s} : \{U_{i}\} \text{ is a }\delta\text{-cover of }F\Big\} \quad \text{and} \quad H^{s}(F)=\lim_{\delta\to 0+}H_{\delta}^{s}(F).$$
		We call $H^{s}(F)$ the $s$-\texttt{dimensional Hausdorff measure} of the set $F$. Using this, the \texttt{Hausdorff dimension} of the set $F$ is defined by 
		$$\dim_{H}(F)=\inf\{s\geq 0 : H^{s}(F)=0\}=\sup\{s \geq 0 : H^{s}(F)=\infty\}.$$
	\end{definition}
	\begin{definition}
		The \texttt{box dimension} of a non-empty bounded subset $F$ of $(X,d)$ is defined as
		$$\dim_{B}F=\lim_{\delta \to 0}\frac{\log{N_{\delta}(F)}}{-\log\delta},$$
		where $N_{\delta}(F)$ denotes the smallest number of sets of diameter at most $\delta$ that can cover $F,$ provided the limit exists. 
		If this limit does not exist, then the upper and lower box dimensions, respectively, are defined as$$\overline{\dim}_{B}F=\limsup_{\delta \to 0}\frac{\log{N_{\delta}(F)}}{-\log\delta}\text{ and }\underline{\dim}_{B}F=\liminf_{\delta \to 0}\frac{\log{N_{\delta}(F)}}{-\log\delta}.$$\end{definition}
        
    \subsection{Symbolic space} Let $\mathcal{I}=\big\{ f_1,f_2,\ldots,f_n \big\}$ be an IFS on $\mathbb{R}^d$ such that $\|f_i(x)-f_i(y)\|\leq r_i\|x-y\|$ with $r_i\in (0,1)$ for all $i\in \{1,2,\dots,n\}$. Let $\Sigma:=\{1,2,\dots,n\}^{\mathbb{N}}$ be the set of all infinite sequences with symbols from $\{1,2,\dots,n\}$. The set $\Sigma$ is the symbolic space corresponding to the IFS $\mathcal{I}=\{f_1,\ldots,f_n\}$. Let $\mathbf{i}=i_{1}i_{2}\dots \in \Sigma$. We define $\mathbf{i}_{|_m}:=i_{1}i_{2}\dots i_{m}$ for all $m\in \mathbb{N}$. We denote the set of all finite sequences of length $m$ with symbols from $\{1,2,\dots, n\}$ by $\Sigma_{m}$. Set $\Sigma^{*}:=\bigcup_{m=1}^\infty\Sigma_{m}.$ The notation $|\mathbf{i}|$ denotes the length of the finite sequence $\mathbf{i}\in \Sigma^*$. 
 The symbolic space $\Sigma$ equipped with metric $\rho$ is a compact metric space, where the metric $\rho$ is defined as follows
$$\rho(\mathbf{i},\mathbf{j})=2^{-|\mathbf{i}\wedge\mathbf{j}|}$$
for $\mathbf{i},\mathbf{j}\in \Sigma$, where $\mathbf{i}\wedge\mathbf{j}$ denotes the initial largest common segment of $\mathbf{i}$ and $\mathbf{j}$. 


\subsection{Affinity dimension and Furstenberg measure}
Let $A$ be a $d \times d$ real matrix. For $t\geq 0$, the \texttt{singular value function} $\Phi^{t}(A)$ of $A$ is defined by 
$$\Phi^{t}(A)=\begin{cases}
    \alpha_{1}\dots \alpha_{\lfloor t \rfloor}{\alpha_{\lceil t \rceil}}^{t- \lfloor t \rfloor}\quad \text{if}~0\leq t\leq d \\
    |\text{det}(A)|^{t/d} \quad \text{if}~t>d,
\end{cases}$$
where $\alpha_{1}\geq \alpha_{2}\geq \dots \geq \alpha_{d}$ are the singular values of $A$.\par 
The \texttt{affinity dimension} of the self-affine IFS $\mathcal{I}=\{f_i(x)=A_ix+a_i\}_{i=1}^{n}$ is defined by 
$$t_0:=\inf\bigg\{t>0: \sum_{m=1}^{\infty}\sum_{i_{1}\dots i_{m}\in \Sigma_m} \Phi^{t}(A_{i_1}\cdots A_{i_m})<\infty\bigg\}.$$

If the matrices $A_i$ have the block triangular form
\begin{equation}\label{eq:mxtype}
A_i=\begin{bmatrix}
    \lambda_iU_i & \underline{0}\\
    \underline{a}_i^T & s_i
\end{bmatrix}\text{ for every $i$},
\end{equation}
where $U_i$ are $2\times2$ orthogonal matrices, $0<\lambda_i<|s_i|<1$, $\underline{a}_i\in\R^2$, and $\mathcal{I}$ satisfies the SOSC then by \cite[Remark~2.6]{BRS2020}, the affinity dimension $t_0$ satisfies the equation $t_0=\min\{r_1,r_2\}$, where
\begin{equation}\label{eq:affindim}
\sum_{i=1}^n|s_i|^{r_1}=1\text{ and }\sum_{i=1}^n|s_i|\lambda_i^{r_2-1}=1.
\end{equation}
In particular, when $\sum_{i=1}^n\lambda_i^2=1$, then $t_0=r_2\in[2,3]$.

Following the lines \cite[Section~2.4]{BRS2020}, we define the corresponding Furstenberg IFS induced by the IFS $\mathcal{I}$ with matrices of the form \eqref{eq:mxtype} as follows:
\begin{equation}\label{eq:furst}
\mathcal{J}=\left\{h_i(\underline{x})=\frac{\lambda_i}{s_i}U_i^T\underline{x}-\frac{1}{s_i}\underline{a}_i\right\}_{i=1}^n.
\end{equation}
The results of Rapaport \cite[Section~1.2]{Rapa2018} and Feng \cite[Theorem~1.10]{Feng} give a sufficient condition to calculate the dimension of the attractor of $\mathcal{I}$, see \cite[Section~2.4]{BRS2020}. We state it in the special case we require throughout this paper.

\begin{theorem}\label{thm:rapaport}
    Let $$\mathcal{I}=\left\{f_i(\underline{x})=\begin{bmatrix}
    \lambda_iU_i & \underline{0}\\
    \underline{a}_i^T & s_i
\end{bmatrix}\underline{x}+a_i\right\}_{i=1}^{n}$$ be a self-affine IFS in $\R^3$ with attractor $K$ such that it satisfies the SOSC, $U_i$ are $2\times2$ orthogonal matrices, $\sum_{i=1}^n\lambda_i^2=1$ and $0<\lambda_i<|s_i|<1$. Let $\mu_F=\sum_{i=1}^n|s_i|\lambda_i^{t_0-1}(h_i)_*\mu_F$ be the Furstenberg measure corresponding to the IFS in \eqref{eq:furst}. If $\dim_H\mu_F\geq3-t_0$ then $\dim_HK=t_0$.
\end{theorem}

Let us point out that Rapaport \cite{Rapa2018} originally proved Theorem~\ref{thm:rapaport} under an irreducibility assumption and strong separation. However, these assumptions can be removed. Feng \cite{Feng} already removed the irreducibility assumption, but still assumed the strong separation. However, using the observation of B\'ar\'any and K\"aenm\"aki \cite[Corollary~2.8]{BK}, we get that the conditional entropy of any Bernoulli measure with respect to the natural projection remains zero under the strong open set condition; thus, the SSC can be replaced by SOSC in \cite[Theorem~1.10(b)]{Feng}.

Finally, we state a simple proposition to estimate the dimension of self-similar measures from below. {The bound is a simpler version of the overlapping number estimate by Mihailescu and Urba\'nski \cite[Theorem~2.5(b)]{MiUr} for random systems.}

\begin{proposition}\label{prop:covering}
    Let $\mathcal{J}=\{h_i(x)=\lambda_iU_ix+t_i\}_{i=1}^N$ be a self-similar IFS on $\R^d$ with attractor $K$ and let $(p_i)_{i=1}^N$ be a non-degenerate probability vector and let $\mu=\sum_{i=1}^Np_i(h_i)_*\mu$. If $\min_{x\in K}\#\{i\in\{1,\ldots,N\}:\ x\notin h_i(K)\}\geq Q$ then $$\dim_H\mu\geq\frac{\log(1-Qp_{\min})}{\log\lambda_{\max}},$$ where $\lambda_{\max}=\max_{i=1,\ldots,N}|\lambda_i|$ and $p_{\min}=\min_{i=1,\ldots,N}p_i$.
\end{proposition}

\begin{proof}
    Let $r\ll1$. Denote $B(x,r)$ the closed ball of radius $r$ with center $x$. Let $C_x:=\{i:\ x\notin h_i(K)\}$. We have 
		\begin{align*}
			\max_{x\in K}\mu(B(x,r))&=\max_{x\in K}\sum_{i=1}^{N}p_i\mu(h_{i}^{-1}(B(x,r)\cap K))\\
			&\leq \max_{x\in K}\sum_{i\in C_{x}^{c}}p_i\mu(B(h_{i}^{-1}(x),\lambda_i^{-1} r))\\&\leq \max_{x\in K}\mu(B(x,\lambda_{\max}^{-1}r))\max_{x\in K}\bigg(1-\sum_{i\in C_{x}}p_i\bigg)\\
            &\leq \max_{x\in K}\mu(B(x,\lambda_{\max}^{-1}r))\max_{x\in K}\bigg(1-Q p_{\min}\bigg).
		\end{align*}
		This implies by induction
        \[
\max_{x\in K}\mu(B(x,\lambda_{\max}^n))\leq(1-Q p_{\min})^n\text{ for every }n\in\N,
        \]
        and so,
        $$\liminf_{r\to0}\frac{\log\mu(B(x,r))}{\log r}=\liminf_{n\to\infty}\frac{\log\mu(B(x,(\lambda_{\max})^n))}{n\log\lambda_{\max}}\geq \frac{\log(1-Qp_{\min})}{\log\lambda_{\max}}.$$
This completes the proof.
\end{proof}

\section{Dimension theory of some self-similar IFS having a common fixed points structure and some negative contraction parameters on line}\label{sec:CFS}

First, we provide techniques to estimate the dimension of the Furstenberg measure from below. Let us now consider a self-similar IFS as follows
$$\mathcal{G}=\{f_{i}(x)=\lambda_{i} x\}_{i=1}^{N_1}\cup \{f_{i}(x)=\lambda_{i} x + \gamma_{i} \lambda_{i}\}_{i=N_1 +1}^{N_2}\cup \{f_{i}(x)=-\lambda_{i} x + \gamma_{i} \lambda_{i}\}_{i=N_2 +1}^{N_3},$$
where $\lambda_{i}\in (0,1)$ for every $i\in \{1,2,\dots,N_1\}$, $\gamma_{i}>0$ for every $i\in \{N_{1} +1, \dots, N_{3}\}$. The IFS considered above has a common fixed point structure. The maps $f_i$ with $i\leq N_1$ share the same fixed point $0$. The authors \cite{BM2025} considered recently such systems. Let us recall some corresponding definitions we need for further analysis.

We denote $I_{0}:=\{1,2,\dots,N_1\}, I_{1}:=\{N_1+1,\dots,N_2\}$ and $I_{2}:=\{N_2+1,\dots,N_3\}$. Set 
$$
D=\min_{k,i\in I_2}\bigg\{\frac{\gamma_{k}}{\gamma_{i}}\bigg\}\text{ and }B=\frac{1}{1+\max_{i\in I_1}\{B_i\}},\text{ where }B_{i}=\max_{k\in I_2}\bigg\{\frac{\gamma_i}{\gamma_k}\bigg\}\text{ for all $i\in I_1$}.
$$
Let $\Sigma$ be the symbolic space corresponding to IFS $\mathcal{G}$. For a symbol $i\in I_0 \cup I_1 \cup I_2 $ and a finite sequence $\mathbf{i}\in\Sigma^*$, let $\#_i\mathbf{i}$ be the number of the appearances of the symbol $i$ in the sequence $\mathbf{i}$. For $\bold{i}\in \Sigma \cup \Sigma^{*}$, we define the ``first block" $b_1^{\mathbf{i}}$ of $\mathbf{i}$ as follows: if $i_1\geq N_1+1$ then $b_1^{\mathbf{i}}=\mathbf{i}_{|_{|b_1^{\mathbf{i}}|}}$ where $|b_1^{\mathbf{i}}|=\min\{k\geq 1: i_{k}\ne i_{1}\}-1$. Otherwise, $b_1^{\mathbf{i}}:=\mathbf{i}_{|_{|b_1^{\mathbf{i}}|}}$ where $|b_1^{\mathbf{i}}|:=\min\{k\geq1:i_k\geq N_1+1\}-1$. Then we define by induction. Suppose that $b_1^{\mathbf{i}},\dots,b_n^{\mathbf{i}}$ are defined and finite. Then let 
$$
|b_{n+1}^{\mathbf{i}}|:=\begin{cases}
    \max\left\{k\geq1:i_{|b_1^{\mathbf{i}}|+\dots+|b_n^{\mathbf{i}}|+1}=i_{|b_1^{\mathbf{i}}|+\dots+|b_n^{\mathbf{i}}|+\ell}~\forall~1\leq\ell\leq k\right\}~ \text{if}~ i_{|b_1^{\mathbf{i}}|+\dots+|b_n^{\mathbf{i}}|+1}> N_{1} \\ \\   \max\left\{k\geq1:i_{|b_1^{\mathbf{i}}|+\dots+|b_n^{\mathbf{i}}|+\ell}\leq N_{1}\text{ for all }1\leq\ell\leq k\right\}~ \text{if}~ i_{|b_1^{\mathbf{i}}|+\dots+|b_n^{\mathbf{i}}|+1}\leq  N_{1}
    \end{cases}
    $$
    If $|b_{n+1}^{\mathbf{i}}|=\left|\sigma^{|b_1^{\mathbf{i}}|+\dots+|b_n^{\mathbf{i}}|}\mathbf{i}\right|$ then let $b_{n+1}^{\mathbf{i}}:=\sigma^{|b_1^{\mathbf{i}}|+\dots+|b_n^{\mathbf{i}}|}\mathbf{i}$, and so, $\mathbf{i}=b_1^{\mathbf{i}}\dots b_n^{\mathbf{i}}b_{n+1}^{\mathbf{i}}$. Otherwise, let $b_{n+1}^{\mathbf{i}}:=(\sigma^{|b_1^{\mathbf{i}}|+\dots+|b_n^{\mathbf{i}}|}\mathbf{i})_{|_{|b_{n+1}^{\mathbf{i}}|}}$. For each $\bold{i}\in \Sigma$ have the following unique block representation 
\begin{equation}\label{blockrep1}
\bold{i}=\underbrace{i_1i_2i_3\dots i_l}_{b_{1}^{\bold{i}}}\underbrace{i_{l+1}\dots i_{m}}_{b_{2}^{\bold{i}}} \underbrace{i_{m+1}\dots i_{n}}_{b_{3}^{\bold{i}}}i_{n+1}\dots  
\end{equation}
 We say that for $\mathbf{i},\mathbf{j}\in\Sigma$, the first blocks are disjoint if the sets formed by the symbols in the first blocks $b_1^{\mathbf{i}}$ and $b_1^{\mathbf{j}}$ are disjoint. We denote it by $b_{1}^{\mathbf{i}}\cap b_{1}^{\mathbf{j}}=\emptyset$. In other words, $\min\{\#_ib_{1}^{\mathbf{i}},\#_ib_{1}^{\mathbf{j}}\}=0$ for every $1\leq i\leq N_3$.
 Let $\Pi$ be the natural projection corresponding to IFS $\mathcal{G}.$

\begin{definition} We say that IFS $\mathcal{G}$ satisfies the \texttt{Exponential Separation Condition for the Common Fixed Point System (ESC for CFS)}, if there exist $N\in \mathbb{N}$ and $b>1$ such that for every $n\geq N$ and every $\mathbf{i}, \mathbf{j}\in \Sigma_n$ with $\lambda_{\mathbf{i}}=\lambda_{\mathbf{j}}$, we have the following:
\begin{equation}\label{eq:ESCCFS}
\text{either}~\mathbf{i} , \mathbf{j}~\text{have the same block structure or}~ |\Pi({\mathbf{i}})-\Pi({\mathbf{j}})|>2^{-b n}.
\end{equation}
\end{definition}

This section aims to show that the IFS $\mathcal{G}$ satisfies the ESC for CFS for typical contraction ratio parameters. The proof follows the lines of \cite[Section~4]{BM2025} with only minor changes. We will only give the essential steps and highlight the differences, but we leave the details for the reader.

\begin{proposition}\label{Exceptionforprojected}
 There exists a set $\mathcal{E}\subset (0,1)^{N_1}\times (0,B)^{N_{2}-N_{1}}\times (0,D)^{N_{3}-N_{2}}$ such that $\dim_{H}\mathcal{E}\leq N_3-1$ such that the IFS $\mathcal{G}$ satisfies ESC for CFS for every parameters $\underline{\lambda}\in (0,1)^{N_1}\times (0,B)^{N_{2}-N_{1}}\times (0,D)^{N_{3}-N_{2}}\setminus\mathcal{E}$. 
\end{proposition}

We begin the discussion with the following, which makes the structure of the IFS slightly less complicated. Still, studying the dimension theory of wedding cake-type surfaces, particularly to estimate the dimension of the corresponding Furstenberg measure, is sufficient.

\begin{lemma}\label{BasicCon}
   Let $\mathcal{G}$ be the self-similar defined as above. If the parameters $\lambda_{i}\in (0,1)~ \forall~ i\in I_{0},$ $\lambda_{i}\in (0,B) ~\forall~ i\in I_{1}$ and $\lambda_{i}\in (0,D) ~\forall~ i\in I_{2}$, then there exists an $A>0$ and an $\tilde{\epsilon}>0$ such that $$f_{i}[0,A]\subset (0,A] ~\forall~ i\in I_1 \cup I_2.$$  
\end{lemma}

\begin{proof}
First, we show that there exists a constant $A>0$ such that $f_{i}[0, A]\subset (0, A]$ for every $i\in I_{1} \cup I_{2}$.  \par
Let us denote the fixed point of the map $f_i$ by ${Fix}(f_i)$. One can see that ${Fix}(f_i)=0$ for every $i\in I_{0}$, ${Fix}(f_i)=\frac{\gamma_{i}\lambda_{i}}{1-\lambda_i}$ for every $i\in I_1$ and ${Fix}(f_i)=\frac{\gamma_{i}\lambda_{i}}{1+\lambda_i}$ for every $i\in I_2$. Since $f_{i}(0)=\gamma_{i}\lambda_{i}$ for every $i\in I_1\cup I_2$, we have 
$$A=\max\left\{\max_{i\in I_2}\gamma_{i}\lambda_{i},\max_{i\in I_1}\frac{\gamma_{i}\lambda_{i}}{1-\lambda_{i}}\right\}.$$
First, we consider $A=\gamma_{i_0}\lambda_{i_0}~\text{for some}~ i_0\in I_2.$ Then $A=\gamma_{i_0}\lambda_{i_0}\geq \frac{\gamma_{k}\lambda_{k}}{1-\lambda_k}$ for every $k\in I_1.$ This implies that $f_{k}[0,A]=[\gamma_{k}\lambda_{k}, \lambda_{k}\gamma_{i_0}\lambda_{i_0}+\gamma_{k}\lambda_{k}]\subseteq(0,A]$ for all $k\in I_1$. For $k\ne i_0 \in I_2$, we have 
$$f_{k}(\gamma_{i_0}\lambda_{i_0})>0\Leftarrow -\lambda_{k} \gamma_{i_0}\lambda_{i_0}+\gamma_{k}\lambda_{k}>0\Leftarrow \lambda_{k} (-\gamma_{i_0}\lambda_{i_0}+\gamma_{k})>0\Leftarrow \lambda_{i_0}<\frac{\gamma_{k}}{\gamma_{i_0}}.$$
 This implies that for the parameters $\lambda_{i}\in (0,1)$ for all $i\in I_0\cup I_1$ and $\lambda_{i}\in (0,D)$ for every $i\in I_2$, then $f_{i}[0,A]\subset (0,A]$ for every $i\in I_1 \cup I_2.$\par 
On the other hand, if $A=\frac{\gamma_{i_0}\lambda_{i_0}}{1-\lambda_{i_0}}~ \text{for some }~ i_0\in I_1$. Then, we have $\frac{\gamma_{i_0}\lambda_{i_0}}{1-\lambda_{i_0}}\geq \frac{\gamma_{k}\lambda_{k}}{1-\lambda_k}$ for all $k\in I_1$ and $\gamma_{k}\lambda_{k}\leq  \frac{\gamma_{i_0}\lambda_{i_0}}{1-\lambda_{i_0}}$ for every $k\in I_2.$ This implies that $f_{k}[0,A]\subseteq(0,A]$ for all $k\in I_1$.
For $k \in I_2$, we have
$$f_{k}(A)>0\Leftarrow-\lambda_{k}\frac{\gamma_{i_0}\lambda_{i_0}}{1-\lambda_{i_0}}+\gamma_{k}\lambda_{k}>0\Leftarrow \lambda_{i_0}< \frac{1}{1+\frac{\gamma_{i_0}}{\gamma_k}}.$$
Thus, for the parameters $\lambda_{i}\in (0,1)$ for every $i\in I_0\cup I_2$ and  $\lambda_{i}\in (0,B)$ for all $i\in I_1$ then $f_{i}[0,A]\subset (0,A]$ for all $i\in I_1 \cup I_2$.
\end{proof}

For every $0<\epsilon\leq \min\{\min_{i\in I_0}\{\lambda_{i}, 1-\lambda_{i}\},\min_{i\in I_1}\{\lambda_{i}, B-\lambda_{i}\}, \min_{i\in I_2}\{\lambda_{i}, D-\lambda_{i}\},\tilde{\epsilon}\}$, then by Lemma \ref{BasicCon}, $$f_{i}[0,A]\subset [{\epsilon},A] ~\forall~ i\in I_{1}\cup I_{2}.$$ 
 
For $\bold{i},\bold{j}\in \Sigma$, we define $\Delta_{\bold{i},\bold{j}}(\underline{\lambda})=\Pi(\bold{i})-\Pi(\bold{j})$ for every vector $\underline{\lambda}\in (0,1)^{N_1}\times (0,B)^{N_{2}-N_{1}}\times (0,D)^{N_{3}-N_{2}}$. 
Let us define the following set of pairs:
\begin{align}
\mathcal{L}:=\bigg\{(\mathbf{i},\mathbf{j})\in \Sigma\times \Sigma: b_1^{\mathbf{i}}\cap b_1^{\mathbf{j}}=\emptyset~\&~b_1^{\mathbf{i}}\neq\mathbf{i}~\&~b_1^{\mathbf{j}}\neq\mathbf{j}\bigg\}.
\end{align}
We divide the set $\mathcal{L}$ further:
\begin{equation}\label{eqBdiv1}
\mathcal{L}_{1}:=\{(\mathbf{i},\mathbf{j})\in \mathcal{L}: i_1\ne j_1,  i_1 \in I_{0}\cup I_1\cup I_{2}, j_1 \in I_1\cup I_{2}\},
\mathcal{L}_2:=\mathcal{L}\setminus\mathcal{L}_1=\left\{(\mathbf{i},\mathbf{j})\in \mathcal{L}: i_1\ne j_1\in I_{0}\right\}.
\end{equation}
Now, set $\tilde{N_{0}}:=\lceil \frac{(1-\epsilon)(2+\epsilon)}{\epsilon^3} \rceil+1$, and divide $\mathcal{L}_2$ further: 
\begin{equation}\label{eqBdiv2}
  \mathcal{L}_{3}=\bigg\{(\mathbf{i},\mathbf{j})\in \mathcal{L}_2:~\max_{k\in I_{0}}\left\{\max\left\{\#_kb_{1}^{\mathbf{i}},\#_kb_{1}^{\mathbf{j}}\right\}\right\}\leq \tilde{N_{0}}\bigg\}\text{ and }\mathcal{L}_4=\mathcal{L}_2\setminus\mathcal{L}_3.
\end{equation}
One can see that $\mathcal{L}_1$ and $\mathcal{L}_3$ are compact subsets of $\Sigma\times\Sigma$.
\begin{lemma}\label{GenExponentialA1}
  Let $\epsilon>0$ be arbitrary as defined above. Then there exists a constant $C>0$ such that for every $\underline{\lambda}\in [\epsilon,1-\epsilon]^{N_1}\times [\epsilon,B-\epsilon]^{N_{2}-N_{1}}\times [\epsilon,D-\epsilon]^{N_{3}-N_{2}}$  and for every $(\mathbf{i},\mathbf{j})\in \mathcal{L}_4$
    $$
    \min\left\{|\Delta_{\mathbf{i},\mathbf{j}}(\underline{\lambda})|,\bigg{|}\frac{\partial\Delta_{\mathbf{i},\mathbf{j}}}{\partial \lambda_{k}}(\underline{\lambda})\bigg{|}\right\}\geq C{\epsilon}^{2\max\{|b_{1}^{\mathbf{i}}|,|b_{1}^{\mathbf{j}}|\}},$$ where $k$ is such that $\max\{\#_kb_{1}^{\mathbf{i}},\#_kb_{1}^{\mathbf{j}}\}>\tilde{N_{0}}$.
\end{lemma}
\begin{proof}
    By Lemma \ref{BasicCon},  the self-similar IFS $\mathcal{G}$ satisfies all the assumptions of the self-similar IFS $\mathcal{F}$ defined in \cite{BM2025} with $c_i=\gamma_{i}\lambda_i$ for all $i\in \{N_{1}+1,\dots,N_3\}$. Thus by \cite[Lemma 3.2]{BM2025}, we get the claim of our result. Let $(\mathbf{i},\mathbf{j})\in \mathcal{L}_4$. This implies that $b_1^{\mathbf{i}}\cap b_1^{\mathbf{j}}=\emptyset$ and $i_1\ne j_1\in I_{0}$.
First, we assume that $\frac{\lambda_{b_{1}^{\mathbf{i}}}}{\lambda_{b_{1}^{\mathbf{j}}}}\notin \bigg(\dfrac{{\epsilon}}{2},\dfrac{2}{{\epsilon}}\bigg)$. Then, by \cite[Lemma 4.1]{BM2025}, we get the
$$
|\Delta_{\mathbf{i},\mathbf{j}}(\underline{\lambda})|\geq {\epsilon}^{2\max\{|b_{1}^{\mathbf{i}}|,|b_{1}^{\mathbf{j}}|\}}.
$$

Lastly, we suppose that $\frac{\lambda_{b_{1}^{\mathbf{i}}}}{\lambda_{b_{1}^{\mathbf{j}}}}\in \bigg(\dfrac{{\epsilon}}{2},\dfrac{2}{{\epsilon}}\bigg)$. Then by \cite[Lemma 4.2]{BM2025}, we get that 
$$
\bigg{|}\frac{\partial\Delta_{\mathbf{i},\mathbf{j}}}{\partial \lambda_{k}}(\underline{\lambda})\bigg{|}\geq C{\epsilon}^{\max\{|b_{1}^{\mathbf{i}}|,|b_{1}^{\mathbf{j}}|\}}
$$
for some uniform constant $C>0$.
\end{proof}

\begin{lemma}\label{Trans2} Let $\epsilon>0$ be arbitrary as defined above. Then there exist $p\geq0$ and $\tilde{C}>0$ such that for every  $(\mathbf{i},\mathbf{j})\in \mathcal{L}_3$ and for all $\underline{\lambda}\in [\epsilon,1-\epsilon]^{N_1}\times [\epsilon,B-\epsilon]^{N_{2}-N_{1}}\times [\epsilon,D-\epsilon]^{N_{3}-N_{2}}$, there exists $(m_{i,j})_{(i,j)\in I}\in\mathbb{N}^{N_3}$ such that $m=\sum_{(i,j)\in I}m_{i,j}\leq p$ and
$$\left|\frac{\partial^m\Delta_{\mathbf{i},\mathbf{j}}}{\prod_{(i,j)\in I}\partial^{m_{i,j}}\lambda_{i,j}}(\underline{\lambda})\right|>\tilde{C}.$$ 
\end{lemma}
\begin{proof}
   The lemma can be proven along the same lines as the proof of \cite[Lemma~4.5]{BM2025}. We omit the details.
\end{proof}

\begin{lemma}\label{Trans3} Let $\epsilon>0$ be arbitrary as defined above. Then there exist $p\geq0$ and $\tilde{C}>0$ such that for every  $(\mathbf{i},\mathbf{j})\in \mathcal{L}_1$ and for all $\underline{\lambda}\in [\epsilon,1-\epsilon]^{N_1}\times [\epsilon,B-\epsilon]^{N_{2}-N_{1}}\times [\epsilon,D-\epsilon]^{N_{3}-N_{2}}$, there exists $(m_{i,j})_{(i,j)\in I}\in\mathbb{N}^{N_3}$ such that $m=\sum_{(i,j)\in I}m_{i,j}\leq p$ and
$$\left|\frac{\partial^m\Delta_{\mathbf{i},\mathbf{j}}}{\prod_{(i,j)\in I}\partial^{m_{i,j}}\lambda_{i,j}}(\underline{\lambda})\right|>\tilde{C}.$$ 
\end{lemma}
\begin{proof}
    The lemma can be proven along the same lines as the proof of \cite[Lemma~4.6]{BM2025}. We leave the details in this case again for the reader.
\end{proof}

\begin{proof}[Proof of Proposition~\ref{Exceptionforprojected}]
    First, we define a set as follows
$$\mathcal{L}^{n}:=\bigg\{(\mathbf{i},\mathbf{j})\in \Sigma_{n}\times \Sigma_{n}: b_1^{\mathbf{i}}\cap b_1^{\mathbf{j}}=\emptyset~\&~b_1^{\mathbf{i}}\neq\mathbf{i}~\&~b_1^{\mathbf{j}}\neq\mathbf{j}\bigg\}.$$
Then for every $\epsilon>0$, let
\begin{equation}\label{eq:E1}
E_\epsilon=\bigcap_{\eta>0 }\bigcap_{\tilde{N}\geq1}\bigcup_{n\geq \tilde{N}}\bigcup_{(\mathbf{i},\mathbf{j})\in\mathcal{L}^{(n)}}\bigg\{\underline{\lambda}\in[\epsilon,1-\epsilon]^{N_1}\times [\epsilon,B-\epsilon]^{N_{2}-N_{1}}\times [\epsilon,D-\epsilon]^{N_{3}-N_{2}}: |\Delta_{\mathbf{i},\mathbf{j}}(\underline{\lambda})|< \eta^{n}\bigg\}.
\end{equation}
Using Lemma \ref{GenExponentialA1}, Lemma \ref{Trans2} and Lemma \ref{Trans3}, and applying the same technique as in \cite[Proposition 4.7]{BM2025}, we get that 
$\dim_{H}({E}_{\epsilon})\leq N_{3}-1$, and for all $\underline{\lambda}\in [\epsilon,1-\epsilon]^{N_1}\times [\epsilon,B-\epsilon]^{N_{2}-N_{1}}\times [\epsilon,D-\epsilon]^{N_{3}-N_{2}} \setminus {E}_{\epsilon},$  $\exists~~ \eta>0, \exists~\tilde{N}\in \mathbb{N}, \forall~n\geq \tilde{N}, \forall~ (\bold{i},\bold{j})\in (\Sigma_{n}\times\Sigma_{n})\cap \mathcal{L}^n$ such that
  $$|\Delta_{\bold{i},\bold{j}}(\underline{\lambda})|> {\eta}^{n}.$$    

We define another set as follows:
\begin{equation*}\label{eq:G1}
G_\epsilon=\bigcup_{n=1}^\infty\bigcup_{\substack{(\mathbf{i},\mathbf{j})\in\Sigma_n\times\Sigma_n\\ \mathbf{i}\cap\mathbf{j}=\emptyset}}\{\underline{\lambda}\in[\epsilon,1-\epsilon]^{N_1}\times [\epsilon,B-\epsilon]^{N_{2}-N_{1}}\times [\epsilon,D-\epsilon]^{N_{3}-N_{2}}:\lambda_{\mathbf{i}}=\lambda_{\mathbf{j}}\}
\end{equation*}
One can show along the lines of \cite[Lemma~4.8]{BM2025} that $\dim_{H}({G}_{\epsilon})\leq N_{3}-1$. We define the exceptional set $\mathcal{E}$ by $\mathcal{E}:={E}\cup G$, where $E=\cup_{n\geq 1}{E}_{1/n}$ and $G=\cup_{n\geq 1}{G}_{1/n}$. Then one can finish the proof by applying the techniques in \cite[Proposition~4.9]{BM2025}.
\end{proof}

\section{Hausdorff dimension of Massopust's surfaces}

This section is devoted to proving Theorem~\ref{Dimalmost} and Theorem~\ref{main2}. Let us recall some definitions from Section~\ref{sec:massopust}. Consider the equilateral triangle $\Delta$ with vertices $\{(0,0), (1,0), (\frac{1}{2},\frac{\sqrt{3}}{2})\},$ and let $\{\Delta_i\}_{i=1}^{N^2}$ be the uniform triangulation for $N\geq 3$. Consider a data set $\{(q_k,a_k)\}_{k=1}^{L(N)}$ associated with the triangulation $\{\Delta_i\}_{i=1}^{N^2}$, where $L(N)=\frac{(N+1)(N+2)}{2}$. We assume that $a_k=0$ for all $k$ such that the corresponding $q_{k}$ is on the boundary of the triangle $\Delta$. 
For each $i\in \{1,2,\dots, N^2\}$, we denote the value at the left vertex of the horizontal line of $\Delta_i$ by $a_{1}^{i}$, value at the right vertex of the horizontal line of $\Delta_i$ by $a_{2}^{i}$ and value at the other vertex by $a_{3}^{i}$. For each $i\in \{1,2,\dots,N^2\}$, define the similarity map $U_{i}: \Delta\to \Delta_i$ as in \eqref{eq:Ui}, the map $$V_i(x,y,z)=\bold{a}_i x+ b_i y+ s_i z+ c_i$$ as in \eqref{eq:Vi} such that it satisfies the boundary condition \eqref{eq:boundcond}.  Clearly, by using the above conditions, we get
  $$\bold{a}_i=-|a_{1}^{i}-a_{2}^{i}|~\text{and}~{b}_{i}=\frac{2}{\sqrt{3}}\bigg(\frac{-(a_{1}^{i}+a_{2}^{i})}{2}+a_{3}^{i}\bigg)$$
Now, we define an affine IFS $\mathcal{I}:=\{W_i,i\in \{1,2,\dots,N^2\}\}$ on $\mathbb{R}^3$, where the map $W_i: \mathbb{R}^3 \to \mathbb{R}^3$ is defined as $$W_i(x,y,z)=(U_i(x,y), V_i(x,y,z)).$$
	Denote $f^*: \Delta \to \mathbb{R}$ the unique fractal interpolation function of which graph $G({f^*})$ is the attractor of the affine IFS $\mathcal{I}.$

    \begin{note}
  One can also see that the subspace generated by the vector $(0,0,1)$ is invariant under the linear parts of the IFS $\mathcal{I}$. Thus, the IFS $\mathcal{I}$ is not strongly irreducible. So, the dimension theory of self-affine IFS presented in \cite{Rapa2024} is not applicable here. 
\end{note}
     
We assume throughout the paper that $s_i\in (\frac{1}{N},1)$. Thus, by \eqref{eq:affindim} the affinity dimension $t_0$ corresponding to the self-affine IFS $\mathcal{I}$ is the unique solution of the following equation 
$$\sum_{i=1}^{N^2}s_{i}\bigg(\frac{1}{N}\bigg)^{t_0-1}=1.$$

By \cite{Falconer1988}, the affinity dimension $t$ is a natural upper bound for the box dimension of the self-affine set, and we have 
\begin{equation}\label{eq:ubmasop}
    \dim_{H}(G(f^*))\leq \overline{\dim}_{B}(G(f^*))\leq t_0=1+\frac{\log(\sum_{i=1}^{N^2}s_i)}{\log N}
\end{equation} for all parameters.\par 

\subsection{The Furstenberg measure and a sufficient condition} Let $\Tilde{\Delta}$ be the interior of the original equilateral triangle $\Delta.$ Then, one can see that 
	$$W_i(\Tilde{\Delta}\times \mathbb{R})\subset \Tilde{\Delta}\times \mathbb{R},~~ W_i(\Tilde{\Delta}\times \mathbb{R})\cap W_j(\Tilde{\Delta}\times \mathbb{R})=\emptyset$$
	for all $i\ne j\in \{1,2,\dots,N^2\}.$ This implies the IFS $\mathcal{I}$ satisfies the SOSC. 
    
    Let $\bold{p}=(p_1,p_2,\dots,p_{N^2})$ be a probability vector, where $p_i=s_i\big(\frac{1}{N}\big)^{t_{0}-1}$ for every $i\in \{1,2,\dots,N^2\}.$

Define 
$$\mathcal{A}_1:=\{i\in\{1,2,\dots,N^2\}: a_{1}^{i}=a_{2}^{i}\},\quad  \mathcal{A}_2:=\{i\in\{1,2,\dots,N^2\}: a_{1}^{i}>a_{2}^{i}\}$$
$$\mathcal{A}_3:=\{i\in\{1,2,\dots,N^2\}: a_{1}^{i}<a_{2}^{i}\}.$$
By applying \eqref{eq:furst}, we construct a Furstenberg IFS $\mathcal{J}=\{h_1,h_2,\dots,h_{N^2}\}$ on $\mathbb{R}^2$ by the self-affine IFS $\mathcal{I}$ as follows
\begin{align*}
    h_i(x,y)=\begin{cases}
    \begin{aligned}
    \frac{1}{Ns_i}\begin{bmatrix}  1 & 0 \\ 0 & 1 \end{bmatrix}\begin{bmatrix} x \\ y \end{bmatrix}-\frac{1}{s_i}\begin{bmatrix} \bold{a}_i  \\ b_i\end{bmatrix} & \text{if}~a_{1}^{i}\geq a_{2}^{i}\text{ and }1\leq i\leq \frac{N(N+1)}{2},\\[4pt]
       \frac{1}{Ns_i}\begin{bmatrix}  1 & 0 \\ 0 & - 1 \end{bmatrix}\begin{bmatrix} x \\ y \end{bmatrix}-\frac{1}{s_i}\begin{bmatrix} \bold{a}_i  \\ b_i\end{bmatrix} & \text{if}~a_{1}^{i}\geq a_{2}^{i}\text{ and }\frac{N(N+1)}{2}+1\leq i\leq N^2,\\[4pt]
       \frac{1}{Ns_i}\begin{bmatrix} - 1 & 0 \\ 0 & 1 \end{bmatrix}\begin{bmatrix} x \\ y \end{bmatrix}-\frac{1}{s_i}\begin{bmatrix} \bold{a}_i \\ b_i\end{bmatrix} & \text{if}~a_{1}^{i}< a_{2}^{i}\text{ and }1\leq i\leq \frac{N(N+1)}{2},\\[4pt]
       \frac{1}{Ns_i}\begin{bmatrix} - 1 & 0 \\ 0 & - 1 \end{bmatrix}\begin{bmatrix} x \\ y \end{bmatrix}-\frac{1}{s_i}\begin{bmatrix} \bold{a}_i \\ b_i\end{bmatrix} & \text{if}~a_{1}^{i}< a_{2}^{i}\text{ and }\frac{N(N+1)}{2}+1\leq i\leq N^2,
    \end{aligned}
    \end{cases}
\end{align*}
The Furstenberg measure $\mu_F=\sum_{i=1}^{N^2}p_i(h_i)_*\mu_F$ is the unique invariant Borel probability measure corresponding to the IFS $\mathcal{J}$ with probability vector $\bold{p}$. Since the self-affine IFS $\mathcal{I}$ satisfies the SOSC and $s_{i}\in (\frac{1}{N},1)~\forall~i\in \{1,2,\dots, N^2\}$, to verify the equality in \eqref{eq:ubmasop}, we only need to show that $$\dim_{H}(\mu_{F})>3-t_{0}$$ by Theorem~\ref{thm:rapaport}. In particular, to show Theorem~\ref{Dimalmost} and Theorem~\ref{main2}, we will prove the following:

\begin{proposition}\label{prop:furstdimae} For each $i\in \{1,2,\dots,N^2\}$, let $a_{1}^{i}\ne a_{2}^{i}$
if $a_{1}^{i}$ and $a_{2}^{i}$ both are not on the boundary of original triangle $\Delta$. Then, 
$$\dim_{H}(\mu_{F})>3-t_0$$
for Lebesgue almost every scaling parameter $\underline{s}\in (\frac{1}{N},1)^{\#\mathcal{A}_1}\times (\frac{1}{N B},1)^{\#\mathcal{A}_2}\times (\frac{1}{N D},1)^{\#\mathcal{A}_3}$. 
\end{proposition}

\begin{proposition}\label{Dimfurstenberg}
		Let	$s_i\in \bigg(\frac{2}{3},1\bigg)~\forall~i\in \{1,2,\dots,9\}$. Suppose that $\max\{s_{5},s_{8}\}\leq \min \{s_{4},s_{7}\}$ and $s_2\leq s_9.$  Then, $$\dim_{H}(\mu_F)>3-t_0,$$
		where $t_0$ is the affinity dimension of the IFS $\mathcal{I}$ in \eqref{eq:ubmasop}.
	\end{proposition}

\begin{proof}[Proof of Theorem~\ref{Dimalmost}]
    The claim follows by \eqref{eq:ubmasop}, and the combination of Theorem~\ref{thm:rapaport} and Proposition~\ref{prop:furstdimae}.
\end{proof}

\begin{proof}[Proof of Theorem~\ref{main2}]
    The claim follows by \eqref{eq:ubmasop}, and the combination of Theorem~\ref{thm:rapaport} and Proposition~\ref{Dimfurstenberg}.
\end{proof}

\begin{remark}
  We note that the method is not applicable in every configuration. In the above construction, if we consider $s_i=s$ for every $i\in \{1,2,\dots,N^2\}$ and $a_k=a$ for all $k\in \{1,2,\dots,L(N)\},$ then for the large value of $N$,
   \begin{equation}\label{notsatisfy}
    \dim_{H}(\mu_{F})\not>3-t_{0}=\frac{-\log s}{\log N}~ \forall~ s\in \bigg(\frac{1}{N},1\bigg).   
   \end{equation}
   In this consideration, there are $9$ different mappings in Furstenberg IFS $\mathcal{J}$ with multiplicity $\frac{(N-3)(N-2)}{2}+3, \frac{(N-4)(N-3)}{2}, (N-2), (N-2), (N-2), (N-2), (N-3), (N-2), 1$, respectively. 
 Examples:  For $N=100, \dim_{H}(\mu_{F})< \frac{-\log s}{\log N}~ \text{for}~s\in (0.042, 0.237).$\\For $N=1000, \dim_{H}(\mu_{F})< \frac{-\log s}{\log N}~ \text{for}~s\in (0.001, 0.430).$ \\For $N=10000, \dim_{H}(\mu_{F})< \frac{-\log s}{\log N}~ \text{for}~s\in (0.0001, 0.461).$\par 
  Although, if we consider some $s_i\ne s_j$ for some $i,j\in \{1,2,\dots,N^2\}$ and $a_k=a$ for all $k\in \{1,2,\dots,L(N)\},$ then for the large value of $N$, the same situation as in \eqref{notsatisfy} occurs. For example, we consider $\frac{(N-3)(N-2)}{2}+3$ many mappings with $s_i=s+0.001,\frac{(N-4)(N-3)}{2}$ many mappings with $s_i=s+0.002, 5(N-2)$ many mappings with $s_i=s+0.003, (N-3)$ many mappings with $s_i=s+0.004$ and $1$ map with $s_i=s+0.002$. In this consideration, we have the followings:\par 
  For $N=100, \dim_{H}(\mu_{F})< 2-\frac{\log\sum_{i=1}^{N^2} s_i}{\log N}~ \text{for}~s\in (0.035, 0.279).$\par
  For $N=100000, \dim_{H}(\mu_{F})< 2-\frac{\log \sum_{i=1}^{N^2} s_i}{\log N}~ \text{for}~s\in (0.00001, 0.463).$
\end{remark}

\begin{remark}
    Let us also note that our method might be applied to other data sets when some of the data values over the horizontal edges of the triangles coincide. Still, there are enough maps where there are no coincidences, and in particular, there are enough maps that do not share the same fixed point, which ensures that the lower bound for the dimension of the Furstenberg measure might hold. 
\end{remark}

\subsection{Dimension for almost every parameter}

Let $f_i$ be the projection of $h_i$ on the $X$-axis. Let $\mathcal{J}_{X}=\{f_1,f_2,\dots,f_{N^2}\}$ be the projection of the IFS $\mathcal{J}$ on the $X$-axis. Then, for $i\in \{1,2,\dots,N^2\}$ the map $f_{i}$ is as follows
$$f_{i}(x)=\begin{cases}
\frac{x}{Ns_i} &\text{if} ~a_{1}^{i}= a_{2}^{i}\\
  \frac{x}{Ns_i}-\frac{\bold{a}_i}{s_i} &\text{if} ~a_{1}^{i}\geq  a_{2}^{i} \\ \frac{-x}{Ns_i}-\frac{\bold{a}_i}{s_i} &\text{if}~a_{1}^{i}<  a_{2}^{i}.
\end{cases}$$
Set $\lambda_{i}=\frac{1}{Ns_i}$ and $\gamma_{i}=-\bold{a}_{i}N=|a_{1}^{i}-  a_{2}^{i}|N>0$, one can see that the IFS $\{f_i\}_{i=1}^{N^2}$ is of type considered in Section~\ref{sec:CFS}. Since $s_i\in (\frac{1}{N},1)$, we have $\lambda_i\in (\frac{1}{N},1)$ for all $i\in \{1,2,\dots,N^2\}.$  
 Thus, we have 
$$\mathcal{J}_{X}=\{f_{i}(x)=\lambda_i x\}_{i\in \mathcal{A}_1}\cup \{f_{i}(x)=\lambda_i x +\gamma_{i} \lambda_i\}_{i\in \mathcal{A}_2}\cup \{f_{i}(x)=-\lambda_i x +\gamma_{i} \lambda_i\}_{i\in \mathcal{A}_3}$$

\begin{proof}[Proof of Proposition~\ref{prop:furstdimae}]
Let ${P_{X*}}\mu_{F}$ be the projection of the measure $\mu_{F}$ on the $X$-axis. The measure ${P_{X*}}\mu_{F}$ is the invariant measure corresponding to the IFS $\mathcal{J}_{X}$ with probability vector $\bold{p}$. By Proposition \ref{Exceptionforprojected}, there exists a set $\mathcal{E}\subset (0,1)^{\#\mathcal{A}_1}\times (0,B)^{\#\mathcal{A}_2}\times (0,D)^{\#\mathcal{A}_3}$ such that $\dim_{H}\mathcal{E}\leq \#\mathcal{A}_1+\#\mathcal{A}_2+\#\mathcal{A}_3-1$ such that the IFS $\mathcal{J}_{X}$ satisfies ESC for CFS for every parameters $\underline{\lambda}\in (0,1)^{\#\mathcal{A}_1}\times (0,B)^{\#\mathcal{A}_2}\times (0,D)^{\#\mathcal{A}_3}\setminus\mathcal{E}$. Thus, by \cite[Theorem 5.1,Theorem 3.5]{BM2025}, we get 
$$\dim_{H}({P_{X*}}\mu_{F})=\min\bigg\{1,\frac{-\sum_{i=1}^{N^2}p_i\log p_i +\Phi(\bold{p})}{-\sum_{i=1}^{N^2}p_i\log \lambda_i}\bigg\}$$
for every parameters $\underline{s}\in (\frac{1}{N},1)^{\#\mathcal{A}_1}\times (\frac{1}{NB},1)^{\#\mathcal{A}_2}\times (\frac{1}{ND},1)^{\#\mathcal{A}_3}\setminus\mathcal{E}$. Now, by \cite[Proposition 2.2]{BM2025}, we have 
$$\Phi(\bold{p})\geq \sum_{i\in \mathcal{A}_1}p_{i}\log\bigg(p_i+\sum_{j\in \mathcal{A}_2\cup \mathcal{A}_3}p_j\bigg).$$
Given that for each $i\in \{1,2,\dots,N^2\}$, $a_{1}^{i}\ne a_{2}^{i}$
if $a_{1}^{i}$ and $a_{2}^{i}$ both are not  on the boundary of original triangle $\Delta$. Under this consideration, one can see that  
$$\#\mathcal{A}_1=N+2~ \text{and}~ \#\mathcal{A}_2+\#\mathcal{A}_3=N^2-N-2.$$
Now, our aim is to show that 
$$\frac{-\sum_{i=1}^{N^2}p_i\log p_i +\Phi(\bold{p})}{-\sum_{i=1}^{N^2}p_i\log \lambda_i}> 3-t_0=2-\frac{\log \sum_{i=1}^{N^2}s_i}{\log N}.$$
Since $\lambda_i=\frac{1}{Ns_i}$ and $p_i=s_i(\frac{1}{N})^{t_0-1}$, we have 
\begin{align*}
\frac{-\sum_{i=1}^{N^2}p_i\log p_i +\Phi(\bold{p})}{-\sum_{i=1}^{N^2}p_i\log \lambda_i}&>2-\frac{\log \sum_{i=1}^{N^2}s_i}{\log N}\\\Leftarrow \frac{-\sum_{i=1}^{N^2}p_i\log p_i +\sum_{i\in \mathcal{A}_1}p_{i}\log\bigg(p_i+\sum_{j\in \mathcal{A}_2\cup \mathcal{A}_3}p_j\bigg)}{\sum_{i=1}^{N^2}p_i\log Ns_i}&>2-\frac{\log \sum_{i=1}^{N^2}s_i}{\log N}\\\Leftarrow \frac{-\sum_{i=1}^{N^2}s_i\log p_i +\sum_{i\in \mathcal{A}_1}s_{i}\log\bigg(p_i+\sum_{j\in \mathcal{A}_2\cup \mathcal{A}_3}p_j\bigg)}{\sum_{i=1}^{N^2}s_i\log Ns_i}&>2-\frac{\log \sum_{i=1}^{N^2}s_i}{\log N}\\\Leftarrow \frac{\sum_{i\in \mathcal{A}_1}s_{i}\log\bigg(\frac{s_i+\sum_{j\in \mathcal{A}_2\cup \mathcal{A}_3}s_j}{s_i}\bigg)+\sum_{i\in \mathcal{A}_2\cup \mathcal{A}_3}s_i\log \bigg(\frac{\sum_{i=1}^{N^2}s_i}{s_i}\bigg) }{\sum_{i=1}^{N^2}s_i\log Ns_i}&>2-\frac{\log \sum_{i=1}^{N^2}s_i}{\log N}\\\Leftarrow \sum_{i\in \mathcal{A}_1}s_{i}\log\bigg(\frac{s_i+\sum_{j\in \mathcal{A}_2\cup \mathcal{A}_3}s_j}{s_i}\bigg)+\sum_{i\in \mathcal{A}_2\cup \mathcal{A}_3}s_i\log \bigg(\frac{\sum_{i=1}^{N^2}s_i}{s_i}\bigg) \\>\sum_{i=1}^{N^2}s_i\log (Ns_i)^2- \sum_{i=1}^{N^2}&s_i\log (Ns_i)\frac{\log \sum_{i=1}^{N^2}s_i}{\log N}\\\Leftarrow \sum_{i\in \mathcal{A}_1}s_{i}\log\bigg(\frac{s_i+\sum_{j\in \mathcal{A}_2\cup \mathcal{A}_3}s_j}{N^2s_{i}^{3}}\bigg)+\sum_{i\in \mathcal{A}_2\cup \mathcal{A}_3}s_i\log \bigg(\frac{\sum_{i=1}^{N^2}s_i}{N^2s_{i}^{3}}\bigg) \\>- \sum_{i=1}^{N^2}&s_i\log Ns_i\frac{\log \sum_{i=1}^{N^2}s_i}{\log N}\\\Leftarrow \frac{\log (N s_{\min})}{\log N}\bigg(\sum_{i\in \mathcal{A}_1}s_{i}\log\bigg(\frac{s_i+\sum_{j\in \mathcal{A}_2\cup \mathcal{A}_3}s_j}{N^2s_{i}^{3}}\sum_{i=1}^{N^2}s_i\bigg)+\sum_{i\in \mathcal{A}_2\cup \mathcal{A}_3}&s_i\log \bigg(\frac{(\sum_{i=1}^{N^2}s_i)^2}{N^2s_{i}^{3}}\bigg)\bigg)>0 \\\Leftarrow \sum_{i\in \mathcal{A}_1}s_{i}\log\bigg(\frac{s_i+\sum_{j\in \mathcal{A}_2\cup \mathcal{A}_3}s_j}{N^2s_{i}^{3}}\sum_{i=1}^{N^2}s_i\bigg)+\sum_{i\in \mathcal{A}_2\cup \mathcal{A}_3}s_i\log \bigg(&\frac{(\sum_{i=1}^{N^2}s_i)^2}{N^2s_{i}^{3}}\bigg)>0.
\end{align*}
Clearly $\frac{(\sum_{i=1}^{N^2}s_i)^2}{N^2s_{i}^{3}}\geq \frac{1}{s_{i}^{3}}>1$. This proves our claim. Since $3-t_0\in (0,1)$, we have 
$$\dim_{H}(\mu_{F})\geq \dim_{H}({P_{X*}}\mu_{F})>3-t_{0}$$
for Lebesgue almost every scaling parameter $\underline{s}\in (\frac{1}{N},1)^{\#\mathcal{A}_1}\times (\frac{1}{N B},1)^{\#\mathcal{A}_2}\times (\frac{1}{N D},1)^{\#\mathcal{A}_3}$. The proof is completed by Theorem~\ref{thm:rapaport}.
\end{proof}

\subsection{Hausdorff dimension of the FIS in the case of $N=3$ for every parameter}
Here, we prove Theorem \ref{main2} by computing the overlapping number. 

Since the upper bound $$\dim_{H}(G(f^*))\leq \overline{\dim}_{B}(G(f^*)) \leq t=1+\frac{\log(\sum_{i=1}^{9}s_i)}{\log(3)}$$ holds for every parameter value, it is enough to show the lower bound.

	For the lower bound, consider the corresponding Furstenberg IFS $\mathcal{J}=\{h_1,h_2,\dots,h_9\}$, where $h_i: \mathbb{R}^2\to \mathbb{R}^2 $ are defined as follows:
	\begin{align*}
		h_1(x,y)=\frac{1}{3s_1}\begin{bmatrix} 1 & 0 \\ 0 & 1 \end{bmatrix}\begin{bmatrix} x \\ y \end{bmatrix}-\frac{1}{s_1}\begin{bmatrix} 0 \\ 0\end{bmatrix}, \quad h_2(x,y)=\frac{1}{3s_2}\begin{bmatrix} 1 & 0 \\ 0 & 1 \end{bmatrix}\begin{bmatrix} x \\ y \end{bmatrix}-\frac{1}{s_2}\begin{bmatrix} 0 \\ \frac{2a}{\sqrt{3}}\end{bmatrix},
	\end{align*}
	\begin{align*}
		h_3(x,y)=\frac{1}{3s_3}\begin{bmatrix} 1 & 0 \\ 0 & 1 \end{bmatrix}\begin{bmatrix} x \\ y \end{bmatrix}-\frac{1}{s_3}\begin{bmatrix} 0 \\ 0\end{bmatrix}, \quad h_4(x,y)=\frac{1}{3s_4}\begin{bmatrix} -1 & 0 \\ 0 & 1 \end{bmatrix}\begin{bmatrix} x \\ y \end{bmatrix}-\frac{1}{s_4}\begin{bmatrix} -a \\ -\frac{a}{\sqrt{3}}\end{bmatrix},
	\end{align*}
	\begin{align*}
		h_5(x,y)=\frac{1}{3s_5}\begin{bmatrix} 1 & 0 \\ 0 & 1 \end{bmatrix}\begin{bmatrix} x \\ y \end{bmatrix}-\frac{1}{s_5}\begin{bmatrix} -a \\ -\frac{a}{\sqrt{3}} \end{bmatrix}, \quad h_6(x,y)=\frac{1}{3s_6}\begin{bmatrix} 1 & 0 \\ 0 & 1 \end{bmatrix}\begin{bmatrix} x \\ y \end{bmatrix}-\frac{1}{s_6}\begin{bmatrix} 0 \\ 0\end{bmatrix},
	\end{align*}
	\begin{align*}
		h_7(x,y)=\frac{1}{3s_7}\begin{bmatrix} -1 & 0 \\ 0 & -1 \end{bmatrix}\begin{bmatrix} x \\ y \end{bmatrix}-\frac{1}{s_7}\begin{bmatrix} -a \\ -\frac{a}{\sqrt{3}} \end{bmatrix}, \quad h_8(x,y)=\frac{1}{3s_8}\begin{bmatrix} 1 & 0 \\ 0 & -1 \end{bmatrix}\begin{bmatrix} x \\ y \end{bmatrix}-\frac{1}{s_8}\begin{bmatrix} -a \\ -\frac{a}{\sqrt{3}}\end{bmatrix},
	\end{align*}
	\begin{align*}
		h_9(x,y)=\frac{1}{3s_9}\begin{bmatrix} 1 & 0 \\ 0 & -1 \end{bmatrix}\begin{bmatrix} x \\ y \end{bmatrix}-\frac{1}{s_9}\begin{bmatrix} 0 \\ \frac{2a}{\sqrt{3}} \end{bmatrix}.
	\end{align*}
       Let $\mu_F=\sum_{i=1}^9s_i\left(\frac{1}{3}\right)^{t-1}(h_i)_*\mu_F$ be the invariant Borel probability measure for the IFS $\mathcal{J}$ with probabilities $p_i=s_i\left(\frac{1}{3}\right)^{t-1}.$ By Theorem~\ref{thm:rapaport}, to show that $$\dim_{H}(G(f^*))\geq t=1+\frac{\log(\sum_{i=1}^{9}s_i)}{\log(3)},$$
       it is enough to prove that $$\dim_{H}(\mu_{F})> 3-t.$$
    
    Now, we will estimate the Hausdorff dimension of the Furstenberg measure.
    
    Let $f_i$ be the projection of $h_i$ on the $X$-axis. Let $\mathcal{J}_{X}=\{f_1,f_2,\dots,f_9\}$ be the projection of the IFS $\mathcal{J}$ on the $X$-axis. Precisely, the maps $f_i's$ are as follows:
	$$f_1(x)=\frac{x}{3s_1}, \quad f_2(x)=\frac{x}{3s_2}, \quad f_3(x)=\frac{x}{3s_3},\quad f_6(x)=\frac{x}{3s_6},\quad f_9(x)=\frac{x}{3s_9},$$
	$$f_4(x)=-\frac{x}{3s_4}+\frac{a}{s_4}, \quad f_5(x)=\frac{x}{3s_5}+\frac{a}{s_5}, \quad f_7(x)=-\frac{x}{3s_7}+\frac{a}{s_7}, \quad f_8(x)=\frac{x}{3s_8}+\frac{a}{s_8}.$$ The fixed point of the map $f_i$ is denoted by $\text{Fix}(f_i)$ for all $i\in \{1,2,\dots,9\}.$ Thus, we have 
	$$\text{Fix}(f_1)=\text{Fix}(f_2)=\text{Fix}(f_3)=\text{Fix}(f_6)=\text{Fix}(f_9)=0,$$
	$$\text{Fix}(f_4)=\frac{3a}{3s_4+1}, \quad \text{Fix}(f_5)=\frac{3a}{3s_5-1}, \quad \text{Fix}(f_7)=\frac{3a}{3s_7+1},\quad \text{Fix}(f_8)=\frac{3a}{3s_8-1}.$$
	Without loss of generality, we assume that $s_5\leq s_8$.
	\begin{lemma}\label{Xaxis} 
		Let $s_i\in \bigg(\frac{2}{3},1\bigg)~~\forall~~i\in \{1,2,\dots,9\}$.  Let $[\Tilde{a},\Tilde{b}]$ be the invariant interval for the IFS $\mathcal{J}_{X}.$ Then, $\Tilde{a}=0$ and $\Tilde{b}=\text{Fix}(f_5).$ Moreover, if $\max\{s_{5},s_{8}\}\leq \min \{s_{4},s_{7}\}$, then
		\begin{equation}\label{eq3.1}
			(f_4[0,\Tilde{b}]\cup f_7[0,\Tilde{b}])\cap  (f_5[0,\Tilde{b}]\cup f_8[0,\Tilde{b}])=\emptyset.  
		\end{equation}
		
	\end{lemma}

    For a visualisation, see Figure~\ref{fig:xaxis}.
    
	\begin{proof} The maps $f_4$ and $f_{7}$ are flipping the orientation and other maps are orientation preserving maps. Since $s_5\leq s_8$, the fixed point $\text{Fix}(f_5)$ is the largest fixed points. Since $s_i\in \bigg(\frac{2}{3},1\bigg)~~\forall~~i\in \{1,2,\dots,9\}$, we have 
$$f_{4}[0,\text{Fix}(f_5)]= \bigg[\frac{a(3s_5-2)}{s_4(3s_5-1)}, \frac{a}{s_4}\bigg]\subset [0, \text{Fix}(f_5)],$$ $$f_{7}[0,\text{Fix}(f_5)]= \bigg[\frac{a(3s_5-2)}{s_7(3s_5-1)}, \frac{a}{s_7}\bigg]\subset [0, \text{Fix}(f_5)].$$ This implies that $[0,\tilde{b}]$ is the invariant interval for the IFS $\mathcal{J}_{X}$, where $\tilde{b}=\text{Fix}(f_5).$ One can see that 
$$f_{5}[0,\tilde{b}]=\bigg[\frac{a}{s_5}, \tilde{b}\bigg]~\text{and}~f_{8}[0,\tilde{b}]=\bigg[\frac{a}{s_8}, \frac{3s_5 a}{s_8 (3s_{5}-1)}\bigg].$$
Now, we assume that $\max\{s_{5},s_{8}\}\leq \min \{s_{4},s_{7}\}$. Then, we get 
$$(f_4[0,\Tilde{b}]\cup f_7[0,\Tilde{b}])\cap  (f_5[0,\Tilde{b}]\cup f_8[0,\Tilde{b}])=\emptyset.$$
This completes the proof.
	\end{proof}
    \begin{figure}
    							\begin{center}
							\begin{tikzpicture}
							\draw[thick, red](0,0)--(4,0)--(6,0);
							\node[above] at (0,0) {0};
								\node[right] at (6,0) {$\text{Fix}(f_5)$};
									\draw[thick, green](4,0.2)--(6,0.2);
										\node[above] at (5,0.2) {$f_5$};
								\draw[thick, blue](3.2,-0.2)--(5,-0.2);
							\node[below] at (4,-0.3) {$f_8$};	
							\draw[thick, black](1,0.2)--(3,0.2);
						\node[above] at (2.0,0.2) {$f_4$};
						\draw[thick,gray ](0.5,-0.2)--(2.5,-0.2);
						\node[below] at (1.5,-0.3) {$f_7$};				
						\draw[thick, red](0,0.1)--(0,-0.1);
							\draw[thick, red](6,0.1)--(6,-0.1);
							\draw[thick, green](6,0.1)--(6,0.3);
							\draw[thick, green](4,0.1)--(4,0.3);
								\draw[thick, blue](3.2,-0.1)--(3.2,-0.3);
									\draw[thick, blue](5,-0.1)--(5,-0.3);
									\draw[thick, black](1,0.1)--(1,0.3);
									\draw[thick, black](3,0.1)--(3,0.3);
									\draw[thick,gray ](0.5,-0.1)--(0.5,-0.3);
									\draw[thick,gray ](2.5,-0.1)--(2.5,-0.3);
							
							\end{tikzpicture}  
                            \caption{Visualisation of the configuration in \eqref{eq3.1}.}\label{fig:xaxis}
						\end{center}
                        \end{figure}
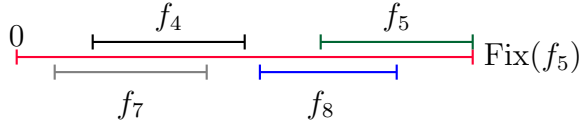
	Now, we will see the projection of the Furstenberg IFS $\mathcal{J}$ on the $Y$-axis. Let $g_i$ be the projection of $h_i$ on the $Y$-axis. Let $\mathcal{J}_{Y}=\{g_1,g_2,\dots,g_9\}$ be the projection of the IFS $\mathcal{J}$ on the $y$-axis. Precisely, the maps $g_i's$ are as follows:
	$$g_1(y)=\frac{y}{3s_1},\quad g_3(y)=\frac{y}{3s_3}, \quad g_6(y)=\frac{y}{3s_6}, $$
	$$g_2(y)=\frac{y}{3s_2}-\frac{2a}{\sqrt{3}s_2}, \quad g_4(y)=\frac{y}{3s_4}+\frac{a}{\sqrt{3}s_4}, \quad g_5(y)=\frac{y}{3s_5}+\frac{a}{\sqrt{3}s_5},$$
	$$g_7(y)=\frac{-y}{3s_7}+\frac{a}{\sqrt{3}s_7},\quad g_8(y)=\frac{-y}{3s_8}+\frac{a}{\sqrt{3}s_8},\quad g_9(y)=\frac{-y}{3s_9}-\frac{2a}{\sqrt{3}s_9}.$$
	Next, we will examine the invariant interval for the IFS $\mathcal{J}_{Y}=\{g_1,g_2,\dots,g_9\}$. The fixed points of the maps $g_i's$ are as follows: 
	$$\text{Fix}(g_1)=\text{Fix}(g_3)=\text{Fix}(g_6)=0, $$
	$$\quad \text{Fix}(g_2)=\frac{-6a}{\sqrt{3}(3s_2-1)}, \quad \text{Fix}(g_4)=\frac{3a}{\sqrt{3}(3s_4-1)}, \quad \text{Fix}(g_5)=\frac{3a}{\sqrt{3}(3s_5-1)},$$
	$$\text{Fix}(g_7)=\frac{3a}{\sqrt{3}(3s_7+1)}, \quad \text{Fix}(g_8)=\frac{3a}{\sqrt{3}(3s_8+1)}, \quad \text{Fix}(g_9)=\frac{-6a}{\sqrt{3}(3s_9+1)}.$$
	Without loss of generality, we assume that $s_7\leq s_8$ and $s_4\leq s_5.$
	\begin{lemma}\label{Yaxis}
		Let $s_i\in (\frac{2}{3},1)\quad \forall\quad i\in \{1,2,\dots,9\}$. We assume that if $\max\{s_{5},s_{8}\}\leq \min \{s_{4},s_{7}\}$ and $s_2\leq s_9.$  Let $\Tilde{I}$ be the invariant interval for the IFS $\mathcal{J}_{Y}$. Then,  the invariant interval $\Tilde{I}$ is either $[\text{Fix}(g_2), \text{Fix}(g_5)]$ or $[\text{Fix}(g_2), g_8(\text{Fix}(g_2))]$.
	\end{lemma}

    For a visualisation, see Figure~\ref{fig:yaxis}.
    
	\begin{proof}
		The maps $g_2,g_4$ and $g_5$ preserve the orientation, however the maps $g_7,g_8$ and $g_9$ are flipping the orientation  about $Y$- axis. The $\text{Fix}(g_2)$ is the lowest fixed point. Since $s_8\leq s_7$ and $s_5\leq s_4,$ the fixed point $\text{Fix}(g_5)$ is the largest fixed point. All the fixed points are always in the invariant interval $\Tilde{I}.$ Thus, $g_8(\text{Fix}(g_2))\in \Tilde{I}.$ And, we have 
		$$g_8(\text{Fix}(g_2))=\frac{-1}{3s_8}\bigg(\frac{-6a}{\sqrt{3}(3s_2-1)}\bigg)+\frac{a}{\sqrt{3} s_8}= \frac{a}{\sqrt{3}s_8}\bigg(\frac{3s_2+1}{3s_2-1}\bigg),$$
		$$g_7(\text{Fix}(g_2))= g_8(\text{Fix}(g_2))\bigg(\frac{s_8}{s_7}\bigg)<g_8(\text{Fix}(g_2)).$$
		Furthermore, $g_8(g_8(\text{Fix}(g_2)))\in \Tilde{I}$ and $g_8(\text{Fix}(g_5))\in \Tilde{I}.$ Thus, we have 
		$$g_8(g_8(\text{Fix}(g_2)))=\frac{a}{\sqrt{3}s_8}\bigg(\frac{3s_8(3s_2-1)-(3s_2+1)}{3s_8(3s_2-1)}\bigg), g_8(\text{Fix}(g_5))=\frac{a}{\sqrt{3}s_8}\bigg(\frac{3s_5-2}{3s_5-1}\bigg).$$
		One can see that $$g_8(g_8(\text{Fix}(g_2)))>\text{Fix}(g_2)~~, 0<g_8(\text{Fix}(g_5))< g_8(\text{Fix}(g_2))~~ \text{and}~~0<g_8(\text{Fix}(g_5))<\text{Fix}(g_5).$$
		The map $g_9$ is also flipping the orientation. So,
		$$g_9(\text{Fix}(g_5))=\frac{-1}{3s_9}\bigg(\frac{3a}{\sqrt{3}(3s_5-1)}\bigg)-\frac{2a}{\sqrt{3} s_9}=\frac{-a}{\sqrt{3}s_9}\bigg(\frac{6s_5-1}{3s_5-1}\bigg),$$
		$$g_9(g_8(\text{Fix}(g_2)))=\frac{-a}{\sqrt{3}s_9}\bigg( \frac{(3s_2+1)+6s_8(3s_2-1)}{3s_8(3s_2-1)}\bigg),$$
		$$g_9(\text{Fix}(g_2))=\frac{-1}{3s_9}\bigg(\frac{-6a}{\sqrt{3}(3s_2-1)}\bigg)-\frac{2a}{\sqrt{3} s_9}=\frac{-2a}{\sqrt{3}s_9}\bigg(\frac{3s_2-2}{3s_2-1}\bigg).$$ 
		One can see that $$0>g_9(\text{Fix}(g_2))>\text{Fix}(g_2).$$
		For $s_2\leq s_9$, we have the following relation for the map $g_9$: 
		$$0>g_9(\text{Fix}(g_5))>\text{Fix}(g_2)\quad \text{and}\quad 0>g_9(g_8(\text{Fix}(g_2)))>\text{Fix}(g_2).$$ 
		Thus, for $s_2\leq s_9$, the invariant interval $\Tilde{I}$ is either $[\text{Fix}(g_2), \text{Fix}(g_5)]$ or $[\text{Fix}(g_2), g_8(\text{Fix}(g_2))]$
		depending on the relation between $\text{Fix}(g_5)$ and $g_8(\text{Fix}(g_2)).$  This completes the proof.
	\end{proof}
    \begin{figure}
    			\begin{center}
				\begin{tikzpicture}
					\draw[thick, red] (0,1.5)--(0,0)--(0,-2.0);
					\node[left] at (0,0) {0};
						\draw[thick, red](0.1,0)--(0,0)--(-0.1,0);
							\draw[thick, red](0.1,-2.0)--(0,-2.0)--(-0.1,-2.0);
								\draw[thick, red](0.1,1.5)--(0,1.5)--(-0.1,1.5);
						\node[left] at (0,-2.0) {$\text{Fix}(g_2)$};
							\node[left] at (0,1.5) {$\text{Fix}(g_5)$};
				\draw[thick, green] (0.5,-0.8)--(0.5,-2.0);	
					\node[right] at (0.5,-1.5) {$g_2$};			
						\draw[thick, green](0.4,-2.0)--(0.5,-2.0)--(0.6,-2.0);
						\draw[thick, green](0.4,-0.8)--(0.5,-0.8)--(0.6,-0.8);
							\draw[thick, blue] (0.5,1.5)--(0.5,-0.6);	
								\node[left] at (0.5,0.5) {$g_5$};	
									\draw[thick, blue](0.4,1.5)--(0.5,1.5)--(0.6,1.5);
										\draw[thick, blue](0.4,-0.6)--(0.5,-0.6)--(0.6,-0.6);
						\draw[thick, gray] (0.7,1.3)--(0.7,-0.5);	
					\node[right] at (0.7,0.5) {$g_4$};		
						\draw[thick, gray](0.6,1.3)--(0.7,1.3)--(0.8,1.3);
					\draw[thick, gray](0.6,-0.5)--(0.7,-0.5)--(0.8,-0.5);				
					\draw[thick, red] (4,1.5)--(4,0)--(4,-2.0);
					\node[left] at (4,0) {0};
					\draw[thick, red](3.9,0)--(4,0)--(4.1,0);
						\draw[thick, red](3.9,1.5)--(4,1.5)--(4.1,1.5);
							\draw[thick, red](3.9,-2.0)--(4,-2.0)--(4.1,-2.0);
					\node[left] at (4,-2.0) {$\text{Fix}(g_2)$};
					\node[left] at (4,1.5) {$g_8(\text{Fix}(g_2))$};
						\draw[thick, green] (4.5,-0.8)--(4.5,-2.0);	
					\node[right] at (4.5,-1.5) {$g_2$};			
					\draw[thick, green](4.4,-2.0)--(4.5,-2.0)--(4.6,-2.0);
					\draw[thick, green](4.4,-0.8)--(4.5,-0.8)--(4.6,-0.8);
						\draw[thick, blue] (4.5,1.5)--(4.5,-0.6);	
					\node[left] at (4.5,0.5) {$g_8$};	
					\draw[thick, blue](4.4,1.5)--(4.5,1.5)--(4.6,1.5);
					\draw[thick, blue](4.4,-0.6)--(4.5,-0.6)--(4.6,-0.6);
					\draw[thick, gray] (4.7,1.3)--(4.7,-0.5);	
					\node[right] at (4.7,0.5) {$g_7$};		
					\draw[thick, gray](4.6,1.3)--(4.7,1.3)--(4.8,1.3);
					\draw[thick, gray](4.6,-0.5)--(4.7,-0.5)--(4.8,-0.5);	
				\end{tikzpicture}  
                \caption{Visualisation of the configuration in Lemma~\ref{Yaxis}.}\label{fig:yaxis}
			\end{center}
            \end{figure}
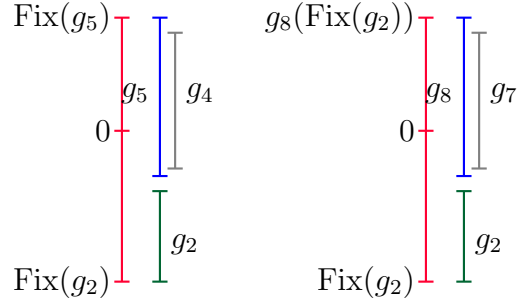

	
	\begin{proposition} \label{localdim}
		Let	$s_i\in \bigg(\frac{2}{3},1\bigg)~\forall~i\in \{1,2,\dots,9\}$.  Suppose that $\max\{s_{5},s_{8}\}\leq \min \{s_{4},s_{7}\}$ and $s_2\leq s_9.$ For $x\in A_{F},$ let $\underline{d}_{\mu_F}(x)$ denote the lower local dimension of $\mu_F$ at $x.$ Then, 
		$$\dim_H\mu_F\geq \frac{\log(1-\frac{s_{\min}}{3^{t_0-2}})}{-\log(3s_{\max})}.$$
	\end{proposition}
	\begin{proof}
    For each $x\in A_{F}$, we define $C_{x}:=\{i: x\notin h_i(A_{F}) ~~\text{and}~~i\in \{1,2,\dots,9\}\}.$ First, we show that $\#C_{x}\geq 3$ for every $x\in A_{F}$, where $\#C_{x}$ denotes the cardinality of $C_{x}.$
    
By Lemma \ref{Xaxis} and Lemma \ref{Yaxis}, either the rectangle $[0,\Tilde{b}]\times [\text{Fix}(g_2), \text{Fix}(g_5)]$ or $[0,\Tilde{b}]\times [\text{Fix}(g_2), g_8(\text{Fix}(g_2))]$ mapped into itself by all the maps of the Furstenberg IFS $\mathcal{J}$.\par
		In the first situation, $\Tilde{I}=[\text{Fix}(g_2), \text{Fix}(g_5)]$. In this case, the cylinder corresponding to $h_2$ is placed at the bottom of the invariant set (on the $Y$-axis) and the cylinder corresponding to $h_5$ is placed at the top of the invariant set (on the left side of the $ Y$-axis). Since $\frac{1}{3s_2}+\frac{1}{3s_5}<1$ and $s_5\leq s_4$, we have  $$h_2([0,\Tilde{b}]\times \Tilde{I})\cap h_5([0,\Tilde{b}]\times \Tilde{I})=\emptyset ~\text{and}~h_2([0,\Tilde{b}]\times \Tilde{I})\cap h_4([0,\Tilde{b}]\times \Tilde{I})=\emptyset.$$ By Equation \ref{eq3.1}, we have  
		$$\bigg(h_4([0,\Tilde{b}]\times \Tilde{I})\cup h_7([0,\Tilde{b}]\times \Tilde{I})\bigg)\cap \bigg(h_5([0,\Tilde{b}]\times \Tilde{I})\cup h_8([0,\Tilde{b}]\times \Tilde{I})\bigg)=\emptyset.$$
		Thus, from the above, it is clear that a point can be contained in at most six cylinders, and so $\# C_x\geq3$. 
        
        In another situation $\Tilde{I}=[\text{Fix}(g_2), g_8(\text{Fix}(g_2))]$, then by using same idea as above, and using the conditions  $\frac{1}{3s_2}+\frac{1}{3s_8}<1$ and $s_8\leq s_7$, one can get $\# C_x\geq 3$. The claim then follows from Proposition~\ref{prop:covering}.
	\end{proof}
	
	\begin{proof}[Proof of Proposition~\ref{Dimfurstenberg}]
		Since $\dim_{H}(\mu_F)= \text{ess inf} (\underline{d}_{\mu_F}(x))$ and by Proposition \ref{localdim}, we obtain
		$$\dim_{H}(\mu_F)\geq \frac{\log(1-\frac{s_{\min}}{3^{t_0-2}})}{-\log(3s_{\max})}.$$
		For proving $\dim_{H}(\mu_F)>3-t_0$, it is enough to show that $\frac{\log(1-\frac{s_{\min}}{3^{t_0-2}})}{-\log(3s_{\max})}>3-t_0.$ We have 
		\begin{align*}
			&\frac{\log(1-\frac{s_{\min}}{3^{t_0-2}})}{-\log(3s_{\max})}>3-t_0\\ \Leftarrow& \frac{\log(1-\frac{2}{3^{t_0-1}})}{-\log(3)}>3-t_0\\ \Leftarrow&
			\bigg(1-\frac{6}{3^{t_0}}\bigg)< \frac{3^{t_0}}{27}\\ \Leftarrow&0< 3^{2t_0}-27\times  3^{t_0} +162.
		\end{align*}
		This implies that if $3^{t_0}> 18 $, then the inequality  $\frac{\log(1-\frac{s_{\min}}{3^{t_0-2}})}{-\log(3s_{\max})}>3-t_0$  holds. Since $s_i\in \bigg(\frac{2}{3},1\bigg)$ and the affinity dimension $t_0=1+\frac{\log(\sum_{i=1}^{9}s_i)}{\log(3)}>1+\frac{\log 6}{\log 3}$, and so, we have $3^{t_0}>18.$ This completes the proof.
	\end{proof}

\begin{remark}
    The method of overlapping numbers might again be applicable for other cases when $N\geq4$ and for general data sets. However, since there are many maps in very general positions, verifying that the overlapping number is sufficiently small would clearly require tedious calculations.
\end{remark}

    \subsection{Uniform scaling factors}
    Next, we will consider a uniform scaling factor $s=s_i~~\forall~~i\in \{1,2,\dots,9\}$ in the construction of fractal surfaces.  For that, we have the following result: 
	\begin{corollary}
		If $s \in (\frac{2}{3},1)$, then $$\dim_{H}(G(f^*))=\dim_{B}(G(f^*))=3+\frac{\log(s)}{\log(3)}.$$
	\end{corollary}
	The proof of this corollary follows from Theorem \ref{main2}.\\
	Next, we show some dimension results for the typical choice of the uniform scaling factor.
	\begin{theorem}\label{TypicalTheorem}
		If $s=s_i~~\forall~~i\in \{1,2,\dots,9\}$, then $$\dim_{H}(G(f^*))=\dim_{B}(G(f^*))=3+\frac{\log(s)}{\log(3)} ~~\text{for a.e.}~~s\in (1/3,1).$$
	\end{theorem}
	\begin{proof} 	If $s=s_i$ for every $i\in \{1,2,\dots,9\}$, then the projected IFS $\mathcal{J}_{X}=\{f_1,f_2,\dots,f_9\}$ of the Furstenberg IFS is as follows. 
		$$f_1(x)=f_2(x)=f_3(x)=f_6(x)=f_9(x)=\frac{x}{3s},$$
		$$f_4(x)=f_7(x)=-\frac{x}{3s}+\frac{a}{s}, \quad f_5(x)=f_8(x)=\frac{x}{3s}+\frac{a}{s}.$$	
		And in this case $p_i=\frac{1}{9}$ for every $i\in \{1,2,\dots,9\}.$ Thus, the IFS $\mathcal{J}_{X}$ is equivalent to the IFS $\tilde{\mathcal{J}_{X}}=\{\tilde{f}_1(x)=\frac{x}{3s}, \tilde{f}_2(x)=-\frac{x}{3s}+\frac{a}{s},  \tilde{f}_3(x)=\frac{x}{3s}+\frac{a}{s}\}$ with probability vector $(\frac{5}{9}, \frac{2}{9},\frac{2}{9}).$ Let ${P_{X*}}\mu_{F}$ be the projection of the measure $\mu_{F}$ on the $X$-axis. Let $\epsilon>0$ and $M$ be a sufficiently large natural number. Now, we define real analytic maps $r_i : [\frac{1}{3}+\epsilon, M]\to (-1,1)\setminus \{0\} \}$ and $d_i : [\frac{1}{3}+\epsilon, M]\to \mathbb{R} \}$ are as follows:
		$$r_1(s)=r_{3}(s)=\frac{1}{3s},~r_{2}(s)=-\frac{1}{3s},~d_1(s)=0,~ d_2(s)=d_3(s)=\frac{a}{s}.$$ The IFS $\tilde{\mathcal{J}_{X}}$ is same as a parametric family of self-similar IFS $\mathcal{I}_{s}=\{r_i(s)x+d_i(s)\}_{i=1}^{3}.$ For $s\in (1,M]$, the parametric IFS $\mathcal{I}_{s}$ satisfies the strong separation condition. Let $A_{s}$ be the attractor of the parametric IFS $\mathcal{I}_{s}$. Let $\Pi_{s}: \Sigma=\{1,2,3\}^{\mathbb{N}}\to A_s$  be the associated natural projection. Then,
		$$\forall~\bold{i},\bold{j}\in \Sigma, \quad \Pi_{s}(\bold{i})=\Pi_{s}(\bold{j})\quad \text{on} \quad \bigg[\frac{1}{3}+\epsilon, M\bigg]\iff \bold{i}=\bold{j}.$$
		Thus by the result of Hochman \cite[Theorem 1.10]{Hochman2014}, we get 
		$$\dim_{H}({P_{X*}}\mu_{F})=\min\bigg\{1,\frac{-\sum_{i=1}^{3}p_i\log(p_i)}{-\sum_{i=1}^{3}p_i\log(|r_i|)}=\frac{\frac{-5}{9}\log(\frac{5}{9})-\frac{4}{9}\log(\frac{2}{9})}{\log(3s)}\bigg\}~\text{for}~a.e. ~s\in \bigg[\frac{1}{3}+\epsilon, M\bigg].$$ The $\epsilon$ is arbitrarily small. Thus, we have 
		$$\dim_{H}(\mu_{F})\geq \dim_{H}({P_{X*}}\mu_{F})=\min\bigg\{1,\frac{\frac{-5}{9}\log(\frac{5}{9})-\frac{4}{9}\log(\frac{2}{9})}{\log(3s)}\bigg\}~\text{for}~a.e. ~s\in \bigg(\frac{1}{3}, 1\bigg).$$
		One can easily see that  $\frac{\frac{-5}{9}\log(\frac{5}{9})-\frac{4}{9}\log(\frac{2}{9})}{\log(3s)}>3-t=\frac{\log(\frac{1}{s})}{\log(3)}$ and $3-t\in (0,1)$ for all $s\in ( \frac{1}{3}, 1).$ Thus,
		$$\dim_{H}(\mu_{F})>3-t~\text{for}~a.e. ~s\in \bigg(\frac{1}{3}, 1\bigg).$$ Thus, by Theorem~\ref{thm:rapaport}, $\dim_{H}(\mu)=t=3+\frac{\log(s)}{\log(3)} ~~\text{for a.e.}~~s\in (1/3,1).$ This completes the proof.
	\end{proof}

    \section{Geronimo-Hardin surfaces}

In this section, we will prove Theorem \ref{GH1} and Theorem \ref{GH2}. 
	 In these results, the self-affine IFS is $\mathcal{I}=\{W_1,W_2,W_3,W_4\}$, where the maps $W_i$ are as follows:
	\begin{equation*}
		W_1(x,y,z)=\begin{bmatrix}  \frac{1}{4} &\frac{\sqrt 3}{4} & 0 \\ \frac{\sqrt 3}{4} & \frac{-1}{4} & 0\\ a & \frac{a}{\sqrt 3} & s \end{bmatrix} \begin{pmatrix} x \\ y \\ z,  \end{pmatrix},\quad  W_2(x,y,z)=\begin{bmatrix}  \frac{1}{4} & \frac{-\sqrt 3}{4} & 0 \\ \frac{-\sqrt 3}{4} & \frac{-1}{4} & 0\\ -a & \frac{a}{\sqrt{3}} & s \end{bmatrix} \begin{pmatrix} x \\ y \\ z  \end{pmatrix}+\begin{pmatrix} \frac{3}{4}\\ \frac{\sqrt 3}{4} \\ a  \end{pmatrix},
	\end{equation*}
	\begin{equation*}
		W_3(x,y,z)=\begin{bmatrix}  \frac{-1}{2} & 0 & 0 \\ 0 & \frac{1}{2} & 0\\ 0 & \frac{-2a}{\sqrt{3}}& s \end{bmatrix} \begin{pmatrix} x \\ y \\ z  \end{pmatrix}+\begin{pmatrix} \frac{3}{4}\\ \frac{\sqrt 3}{4} \\ a  \end{pmatrix}, \quad W_4(x,y,z)=\begin{bmatrix}  \frac{-1}{2} & 0 & 0 \\ 0 & \frac{-1}{2} & 0\\ 0 & 0 & s \end{bmatrix} \begin{pmatrix} x \\ y \\ z  \end{pmatrix}+\begin{pmatrix} \frac{3}{4}\\ \frac{\sqrt 3}{4} \\ a  \end{pmatrix}.
	\end{equation*}
	Thus, the IFS $\mathcal{I}$ is a block triangular self-affine IFS. Let $s\in (\frac{1}{2},1).$ Then, the affinity dimension $(t)$ for the IFS $\mathcal{I}$ is uniquely given by the following equation $$\sum_{i=1}^{4}s\bigg(\frac{1}{2}\bigg)^{t-1}=1,$$ and by \cite{Falconer1988}, we get 
	\begin{equation}\label{eq:GHub}
    \dim_{H}(G(f^*))\leq \overline{\dim}_{B}(G(f^*)) \leq t=3+\frac{\log(s)}{\log(2)}.
    \end{equation}

	 Since bi-Lipschitz conjugation preserves the fractal dimension, we may assume without loss of generality that $a=1$.
    For the lower bound, again, we construct the corresponding Furstenberg IFS $\mathcal{J}=\{h_1,h_2,h_3,h_4\}$ by \eqref{eq:furst}, where $h_i: \mathbb{R}^2\to \mathbb{R}^2 $ are defined as follows:
\begin{equation}\label{eq:furstgh}
\begin{gathered}
		h_1(x,y)=\frac{1}{2s}\begin{bmatrix} \frac{1}{2} & \frac{\sqrt 3}{2} \\ \frac{\sqrt 3}{2} & \frac{-1}{2} \end{bmatrix}\begin{bmatrix} x \\ y \end{bmatrix}+\begin{bmatrix} 1 \\ \frac{1}{\sqrt 3}\end{bmatrix}, \quad h_2(x,y)=\frac{1}{2s}\begin{bmatrix} \frac{1}{2} & \frac{-\sqrt 3}{2} \\ \frac{-\sqrt 3}{2} & \frac{-1}{2} \end{bmatrix}\begin{bmatrix} x \\ y \end{bmatrix}+\begin{bmatrix} -1 \\ \frac{1}{\sqrt 3}\end{bmatrix},\\
		h_3(x,y)=\frac{1}{2s}\begin{bmatrix} -1 & 0 \\ 0 & 1 \end{bmatrix}\begin{bmatrix} x \\ y \end{bmatrix}+\begin{bmatrix} 0 \\ \frac{-2}{\sqrt 3}\end{bmatrix}, \quad h_4(x,y)=\frac{1}{2s}\begin{bmatrix} -1 & 0 \\ 0 & -1 \end{bmatrix}\begin{bmatrix} x \\ y \end{bmatrix}.
	\end{gathered}
\end{equation}
	For $s\in (\frac{1}{2},1),$ the Furstenberg IFS $\mathcal{J}$ is contractive. Let $A_{F}$ be the attractor of the Furstenberg IFS $\mathcal{J}$. The  Furstenberg measure $\mu_F$ is the invariant Borel probability measure for the IFS $\mathcal{J}$ with the uniform probability vector $p=(\frac14,\frac14,\frac14,\frac14).$  By taking bi-Lipschitz conjugate  of the IFS $\mathcal{J}$  with the map $g(x,y)=\frac{-a}{s}(x,y)$, the resultant IFS is denoted by $\mathcal{J}=\{h_1,h_2,h_3,h_4\}$, which is as follows.

    \subsection{Every-type result}
	We can not proceed with the argument as in the previous section, because the projection of the Furstenberg IFS $\mathcal{J}$ on the $X$-axis and $Y$-axis is not an IFS. In this case, we need to directly determine the invariant set for the IFS $\mathcal{J}$ and with the help of that, we estimate the overlapping number for the IFS $\mathcal{J}$. The fixed points of the maps $h_i$'s are as follows.
	$$\text{Fix}(h_1)=\bigg(\frac{4s}{4s-2}, \frac{4s}{\sqrt {3}(4s-2)}
	\bigg), \quad \text{Fix}(h_2)=\bigg(-\frac{4s}{4s-2}, \frac{4s}{\sqrt {3}(4s-2)}
	\bigg),$$
	$$\text{Fix}(h_3)=\bigg(0, -\frac{8s}{\sqrt {3}(4s-2)}
	\bigg), \quad \text{Fix}(h_4)=(0,0).$$
	\begin{proposition}\label{overNUM2}
		Let $s\in  \bigg[\frac{1+\sqrt{5}}{4},1\bigg).$ Then, $\min_{x\in K}\#\{i\in\{1,\ldots,N\}:\ x\notin h_i(A_F)\}\geq 1.$ 
	\end{proposition}
	\begin{proof}
		Let $\text{Fix}(h_1)=A, \text{Fix}(h_2)=B$ and $\text{Fix}(h_3)=C$.  Let $\mathcal{S}$ be the invariant set for the IFS $\mathcal{J}$.  The map $h_4$ flips the orientation about the $X$-axis. Since $A,B,C\in \mathcal{S},$ we have $A'=h_4(A)=\bigg(-\frac{2}{4s-2},-\frac{2}{\sqrt{3}(4s-2)}\bigg)\in \mathcal{S}, B'=h_4(B)=\bigg(\frac{2}{4s-2},-\frac{2}{\sqrt{3}(4s-2)}\bigg)\in  \mathcal{S}$ and $C'=h_4(C)=\bigg(0,\frac{4}{\sqrt{3}(4s-2)}\bigg)\in  \mathcal{S}.$ For $s\in (\frac{1}{2},1),$ our claim is that the convex set with vertices $A,B,C,A',B'$ and $C'$ is the invariant set for the IFS $\mathcal{J}$. We denote that convex set by $\text{Con}(ABCA'B'C').$  One can see that the set $\text{Con}(ABCA'B'C')$ is symmetric about $Y$-axis, see Figure~\ref{fig:convhull}.
		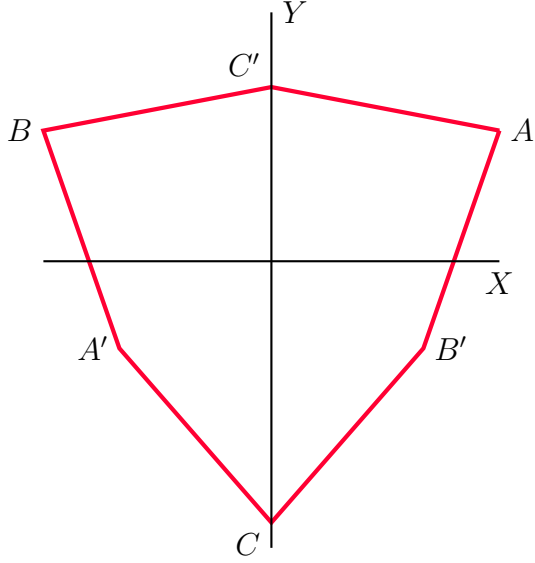
\begin{figure}[h!]
        \begin{center}
			\begin{tikzpicture}
				\draw[ultra thick, red](3, sqrt 3)--(0, 4/ sqrt 3)--(-3, sqrt 3)--(-2,-2/ sqrt 3)--(0, - sqrt 12)--(2,- 2/ sqrt 3)--(3, sqrt 3);
				\node[above left] at (0, 4/ sqrt 3) {$C'$};
				\node[right] at (3, sqrt 3) {$A$};
				\node[left] at (-3, sqrt 3) {$B$};
				\node[left] at (-2,-2/ sqrt 3) {$A'$};
				\node[right] at (2,-2/ sqrt 3) {$B'$};
				\node[below left] at (0, - sqrt 12) {$C$};
				\draw[thick, black](-3,0)--(0,0)--(3,  0);	
				\draw[thick, black](0,-3.8)--(0,0)--(0,  3.3);				
				\node[below] at (3,0) {$X$};
				\node[right] at (0,3.3) {$Y$};
			\end{tikzpicture}  
            \caption{The invariant convex hull of the Furstenberg IFS \eqref{eq:furstgh}.}\label{fig:convhull}
		\end{center}
        \end{figure}
        
		Now, we prove our claim. We have 
		$$h_1(A)=A,~~h_1(B)=\bigg(1, \frac{4s-6}{\sqrt{3}(4s-2)}\bigg)=B_{1},~~ h_1(C)=\bigg(\frac{4s-4}{4s-2}, \frac{4s}{\sqrt{3}(4s-2)}\bigg)=C_{1},$$
		$$h_{1}(A')=\bigg(\frac{-1+s(4s-2)}{s(4s-2)}, \frac{-1+s(4s-2)}{s\sqrt{3}(4s-2)}\bigg)=A_{1}', h_1(B')=\bigg(1, \frac{2+s(4s-2)}{s\sqrt{3}(4s-2)}\bigg)=B_{1}'$$
		$$h_1(C')=\bigg(\frac{1+s(4s-2)}{s(4s-2)}, \frac{-1+s(4s-2)}{s\sqrt{3}(4s-2)}\bigg)=C_{1}'.$$
		One can see that the points $B_{1}'$ and $C_{1}'$ are on the line $AC'$	and $AB'$, respectively. For $s\in (\frac{1}{2},1),$  we get $1<\frac{2}{(4s-2)}.$ And the point $\bigg(1, \frac{(4s-1)(4s-2)-8s}{\sqrt{3}(4s-2)}\bigg)$ is on the line $CB'$. Since $ -\frac{2}{\sqrt{3}(4s-2)}>\frac{4s-6}{\sqrt{3}(4s-2)}> \frac{(4s-1)(4s-2)-8s}{\sqrt{3}(4s-2)},$ the point $B_1$ is on the above side of line $CB'$.  For $s\in (\frac{1}{2},1),$ we have $-\frac{2}{4s-2}< \frac{4s-4}{4s-2}<0$. Thus the point $C_1$ is in the left side of the set  $\text{Con}(ABCA'B'C')$ and on the line joining $A$ and $B$. This implies that $h_1(\text{Con}(ABCA'B'C'))\subset \text{Con}(ABCA'B'C').$ For the mapping $h_2$, we have 
		$$h_2(B)=B,~~h_2(A)=\bigg(-1, \frac{4s-6}{\sqrt{3}(4s-2)}\bigg)=A_{2},~~ h_2(C)=\bigg(-\frac{4s-4}{4s-2}, \frac{4s}{\sqrt{3}(4s-2)}\bigg)=C_{2},$$
		$$h_{2}(B')=\bigg(-\frac{-1+s(4s-2)}{s(4s-2)}, \frac{-1+s(4s-2)}{s\sqrt{3}(4s-2)}\bigg)=B_{2}', h_2(A')=\bigg(-1, \frac{2+s(4s-2)}{s\sqrt{3}(4s-2)}\bigg)=A_{2}'$$
		$$h_2(C')=\bigg(-\frac{1+s(4s-2)}{s(4s-2)}, \frac{-1+s(4s-2)}{s\sqrt{3}(4s-2)}\bigg)=C_{2}'.$$ Thus, one can see that $h_2(\text{Con}(ABCA'B'C'))$ is the mirror image of  $h_1(\text{Con}(ABCA'B'C'))$ with respect to $Y$-axis. Thus, due to symmetry of the set $\text{Con}(ABCA'B'C')$ with respect to the $Y$-axis, we get $h_2(\text{Con}(ABCA'B'C'))\subset \text{Con}(ABCA'B'C').$  For the mapping $h_3$, we have
		$$h_3(C)=C,~~h_3(C')=\bigg(0, \frac{2(1-s(4s-2))}{\sqrt{3}s(4s-2)}\bigg)=C_{3}',$$
		$$h_{3}(A)=\bigg(\frac{-2}{4s-2}, \frac{6-8s}{\sqrt{3}(4s-2)}\bigg)=A_{3}, h_3(A')=\bigg(\frac{1}{s(4s-2)}, \frac{-1-2s(4s-2)}{\sqrt{3} s(4s-2)}\bigg)=A_{3}',$$
		$$h_3(B)=\bigg(\frac{2}{4s-2}, \frac{6-8s}{\sqrt{3}(4s-2)}\bigg)=B_{3} ,h_3(B')=\bigg(\frac{-1}{s(4s-2)}, \frac{-1-2s(4s-2)}{\sqrt{3} s(4s-2)}\bigg)=B_{3}'.$$ 
		One can see that the points $A_{3}'$ and $B_{3}'$ are on the line $CB'$ and $CA'$, respectively. For $s\in (\frac{1}{2}, 1),$ we have $\frac{-2}{\sqrt{3}(4s-2)}< \frac{6-8s}{\sqrt{3}(4s-2)}< \frac{-1-2s(4s-2)}{\sqrt{3} s(4s-2)}$ and 
		$\frac{6-8s}{\sqrt{3}(4s-2)}< \frac{-1-2s(4s-2)}{\sqrt{3} s(4s-2)}< \frac{4}{\sqrt{3}(4s-2)}.$ This implies that $h_3(\text{Con}(ABCA'B'C'))\subset \text{Con}(ABCA'B'C').$ 
		For the mapping $h_4$, we have
		$$h_4(A)=A', h_4(B)=B', h_4(C)=C',$$
		$$h_4(A')=\bigg(\frac{1}{s(4s-2)}, \frac{1}{\sqrt{3} s(4s-2)}\bigg)=A_{4}', h_4(B')=\bigg(\frac{-1}{s(4s-2)}, \frac{1}{\sqrt{3} s(4s-2)}\bigg)=B_{4}',$$
		$$h_4(C')=\bigg(0,\frac{-2}{\sqrt{3}s(4s-2)}\bigg).$$
		For $s\in (\frac{1}{2},1),$ we have $\frac{-8s}{\sqrt{3}(4s-2)}<\frac{-2}{\sqrt{3}s(4s-2)}<\frac{-2}{\sqrt{3}(4s-2)}, 0<\frac{1}{s(4s-2)}<\frac{2}{4s-2}$ and $0<\frac{1}{\sqrt{3}s(4s-2)}<\frac{2}{\sqrt{3}(4s-2)}.$ This implies that $h_4(\text{Con}(ABCA'B'C'))\subset \text{Con}(ABCA'B'C').$ Thus the set $\text{Con}(ABCA'B'C')$ is an invariant set for the IFS $\mathcal{J}.$ \par Now, we will estimate the overlapping number. For $s=\frac{1+\sqrt{5}}{4},$ one can see that  $A_{1}'=B_{2}'=C_{3}'=(0,0)$, $h_1(\text{Con}(ABCA'B'C'))$ is on the right side, $h_2(\text{Con}(ABCA'B'C'))$ is one the left side (mirror image of $h_1(\text{Con}(ABCA'B'C'))$ with respect to $Y$-axis) and  $h_3(\text{Con}(ABCA'B'C'))$ is below the $X$-axis. Thus, only at $(0,0)$, the overlapping number is $4$. This implies that $K\leq 3$ a.e. $x\in A_{F}.$ For a visualisation, see Figure~\ref{fig:firstlevel}.
        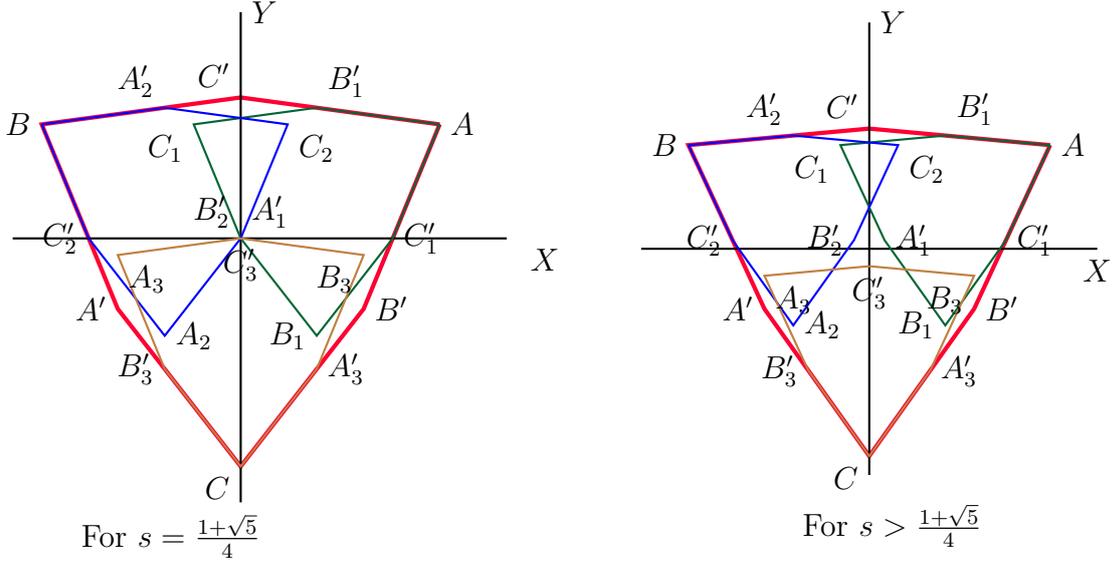
\begin{figure}[h!]
	\begin{minipage}{0.5\textwidth}
    	\begin{tikzpicture}
				\draw[ultra thick, red](3/2+ sqrt 5/2, 3/ sqrt 12 + sqrt 5/ sqrt 12 )--(0, sqrt 5/sqrt 3+1/sqrt 3)--(-3/2- sqrt 5/2, 3/ sqrt 12 + sqrt 5/ sqrt 12 )--(-sqrt 5/2-1/2, -sqrt 5/ sqrt 12-1/ sqrt 12)--(0, -3/ sqrt 3- sqrt 5 /sqrt 3)--(sqrt 5/2+1/2, -sqrt 5/ sqrt 12-1/ sqrt 12)--(3/2+ sqrt 5/2, 3/ sqrt 12 + sqrt 5/ sqrt 12 );
				\node[above left] at (0, sqrt 5/sqrt 3+1/sqrt 3) {$C'$};
				\node[right] at (3/2+ sqrt 5/2, 3/ sqrt 12 + sqrt 5/ sqrt 12 ) {$A$};
				\node[left] at (-3/2- sqrt 5/2, 3/ sqrt 12 + sqrt 5/ sqrt 12 ) {$B$};
				\node[left] at (-sqrt 5/2-1/2, -sqrt 5/ sqrt 12-1/ sqrt 12) {$A'$};
				\node[right] at (sqrt 5/2+1/2, -sqrt 5/ sqrt 12-1/ sqrt 12) {$B'$};
				\node[below left] at (0, -3/ sqrt 3- sqrt 5 /sqrt 3) {$C$};
				\draw[thick, black](-3,0)--(0,0)--(3.5,  0);	
				\draw[thick, black](0,-3.5)--(0,0)--(0,  3);				
				\node[below] at (4,0) {$X$};
				\node[right] at (0,3) {$Y$};
				\draw[thick, green](3/2+ sqrt 5/2, 3/ sqrt 12 + sqrt 5/ sqrt 12 )--(1, sqrt 3)--(1/2- sqrt 5/2, 3/ sqrt 12 + sqrt 5/ sqrt 12 )--(0, 0)--(1, -sqrt 5/ sqrt 3)--(2, 0)--(3/2+ sqrt 5/2, 3/ sqrt 12 + sqrt 5/ sqrt 12 );
				\node[above right] at (1, sqrt 3) {$B_{1}'$};
				\node[below left] at (1/2- sqrt 5/2, 3/ sqrt 12 + sqrt 5/ sqrt 12 ) {$C_{1}$};
				\node[above right] at (0, 0) {$A_{1}'$};
				\node[left] at (1, -sqrt 5/ sqrt 3) {$B_{1}$};
				\node[right] at (2, 0) {$C_{1}'$};
				\draw[thick, blue](-3/2- sqrt 5/2, 3/ sqrt 12 + sqrt 5/ sqrt 12 )--(-1, sqrt 3)--(-1/2+ sqrt 5/2, 3/ sqrt 12 + sqrt 5/ sqrt 12 )--(0, 0)--(-1, -sqrt 5/ sqrt 3)--(-2, 0)--(-3/2- sqrt 5/2, 3/ sqrt 12 + sqrt 5/ sqrt 12 );
				\node[above left] at (-1, sqrt 3) {$A_{2}'$};
				\node[below right] at (-1/2+ sqrt 5/2, 3/ sqrt 12 + sqrt 5/ sqrt 12 ) {$C_{2}$};
				\node[above left] at (0, 0) {$B_{2}'$};
				\node[right] at (-1, -sqrt 5/ sqrt 3) {$A_{2}$};
				\node[left] at (-2, 0) {$C_{2}'$};
				\draw[thick, brown](0, 0)--(-1/2- sqrt 5/2, sqrt 5/ sqrt 12 - sqrt 3/ 2)--(-1, -sqrt 3)--(0, -3/ sqrt 3- sqrt 5 /sqrt 3)--(1, -sqrt 3)--(1/2+ sqrt 5/2, sqrt 5/ sqrt 12 - sqrt 3/ 2)--(0, 0);
				\node[below] at (0, 0) {$C_{3}'$};
				\node[below right] at (-1/2- sqrt 5/2, sqrt 5/ sqrt 12 - sqrt 3/ 2) {$A_{3}$};
				\node[below left] at (1/2+ sqrt 5/2, sqrt 5/ sqrt 12 - sqrt 3/ 2) {$B_{3}$};
				\node[left] at (-1, -sqrt 3) {$B_{3}'$};
				\node[right] at (1, -sqrt 3) {$A_{3}'$};
			\end{tikzpicture}
		{\hspace*{1cm} For $s=\frac{1+\sqrt{5}}{4}$}
	\end{minipage}\hspace*{0.3cm}
	\begin{minipage}{0.5\textwidth}	\begin{tikzpicture} 
				\draw[ultra thick, red](7/5+ sqrt 24/5, 7/ sqrt 75 + sqrt 24/ sqrt 75 )--(0,  sqrt 96/sqrt 75+4/sqrt 75)--(-7/5- sqrt 24/5, 7/ sqrt 75 + sqrt 24/ sqrt 75 )--(-sqrt 24/5-2/5, -sqrt 24/ sqrt 75-2/ sqrt 75)--(0, -14/ sqrt 75- sqrt 96 /sqrt 75)--(sqrt 24/5+2/5, -sqrt 24/ sqrt 75-2/ sqrt 75)--(7/5+ sqrt 24/5, 7/ sqrt 75 + sqrt 24/ sqrt 75 );
				\node[above left] at (0,  sqrt 96/sqrt 75+4/sqrt 75) {$C'$};
				\node[right] at (7/5+ sqrt 24/5, 7/ sqrt 75 + sqrt 24/ sqrt 75 ) {$A$};
				\node[left] at (-7/5- sqrt 24/5, 7/ sqrt 75 + sqrt 24/ sqrt 75 ) {$B$};
				\node[left] at (-sqrt 24/5-2/5, -sqrt 24/ sqrt 75-2/ sqrt 75) {$A'$};
				\node[right] at (sqrt 24/5+2/5, -sqrt 24/ sqrt 75-2/ sqrt 75) {$B'$};
				\node[below left] at (0, -14/ sqrt 75- sqrt 96 /sqrt 75) {$C$};
				\draw[thick, black](-3,0)--(0,0)--(3,  0);	
				\draw[thick, black](0,-3)--(0,0)--(0,  3);				
				\node[below] at (3,0) {$X$};
				\node[right] at (0,3) {$Y$};
				\draw[thick, green](7/5+ sqrt 24/5, 7/ sqrt 75 + sqrt 24/ sqrt 75 )--(1, 13/sqrt 75)--(3/5- sqrt 24/5, 7/ sqrt 75 + sqrt 24/ sqrt 75 )--(1/5, 1/ sqrt 75)--(1, 1/ sqrt 75- sqrt 96/sqrt 75)--(9/5, 1/sqrt 75)--(7/5+ sqrt 24/5, 7/ sqrt 75 + sqrt 24/ sqrt 75 );
				\node[above right] at (1, 13/sqrt 75) {$B_{1}'$};
				\node[below left] at (3/5- sqrt 24/5, 7/ sqrt 75 + sqrt 24/ sqrt 75 ) {$C_{1}$};
				\node[right] at (1/5, 1/ sqrt 75) {$A_{1}'$};
				\node[left] at (1, 1/ sqrt 75- sqrt 96/sqrt 75) {$B_{1}$};
				\node[right] at (9/5, 1/sqrt 75) {$C_{1}'$};
				\draw[thick, blue](-7/5- sqrt 24/5, 7/ sqrt 75 + sqrt 24/ sqrt 75 )--(-1, 13/sqrt 75)--(-3/5+ sqrt 24/5, 7/ sqrt 75 + sqrt 24/ sqrt 75 )--(-1/5, 1/ sqrt 75)--(-1, 1/ sqrt 75- sqrt 96/sqrt 75)--(-9/5, 1/sqrt 75)--(-7/5- sqrt 24/5, 7/ sqrt 75 + sqrt 24/ sqrt 75 );
				\node[above left] at (-1, 13/sqrt 75) {$A_{2}'$};
				\node[below right] at (-3/5+ sqrt 24/5, 7/ sqrt 75 + sqrt 24/ sqrt 75 ) {$C_{2}$};
				\node[left] at (-1/5, 1/ sqrt 75) {$B_{2}'$};
				\node[right] at (-1, 1/ sqrt 75- sqrt 96/sqrt 75) {$A_{2}$};
				\node[left] at (-9/5, 1/sqrt 75){$C_{2}'$};
				\draw[thick, brown](0, -2/sqrt 75)--(-2/5- sqrt 24/5, sqrt 24/ sqrt 75 - 8/ sqrt 75)--(-4/5, -14/ sqrt 75)--(0, -14/ sqrt 75- sqrt 96 /sqrt 75)--(4/5, -14/ sqrt 75)--(2/5+ sqrt 24/5, sqrt 24/ sqrt 75 - 8/ sqrt 75)--(0, -2/sqrt 75);
				\node[below] at (0, -2/sqrt 75) {$C_{3}'$};
				\node[below right] at (-2/5- sqrt 24/5, sqrt 24/ sqrt 75 - 8/ sqrt 75) {$A_{3}$};
				\node[below left] at (2/5+ sqrt 24/5, sqrt 24/ sqrt 75 - 8/ sqrt 75){$B_{3}$};
				\node[left] at (-4/5, -14/ sqrt 75) {$B_{3}'$};
				\node[right] at (4/5, -14/ sqrt 75) {$A_{3}'$};
			\end{tikzpicture}  
		{\hspace*{2cm} For $s>\frac{1+\sqrt{5}}{4}$}
	\end{minipage}	
    \caption{The first level cylinder sets for the Furstenberg IFS \eqref{eq:furstgh}.}\label{fig:firstlevel}
    \end{figure}

        Now, we assume that  $s\in (\frac{1+\sqrt{5}}{4},1)$. In this case, we have  $s(4s-2)>1.$ Thus, $A_{1}'$ and $B_{1}'$ are in the $1$st and $2$nd quadrant of the set $\text{Con}(ABCA'B'C'),$ respectively. And $C_{3}'$ is on negative $Y$-axis. The points $C_{1}'$ and $C_{2}'$ are on the above of the $X$-axis and on line $AB'$ and $BA'$, respectively. Thus, any point of $\text{Con}(ABCA'B'C')$ can lie at most two of the sets $h_1(\text{Con}(ABCA'B'C'))$, $h_2(\text{Con}(ABCA'B'C'))$ and $h_3(\text{Con}(ABCA'B'C'))$. Thus, $K\leq 3.$ This completes the proof.
	\end{proof}
	
	\begin{proof}[Proof of Theorem~\ref{GH1}]
		Combining Proposition~\ref{prop:covering} and Proposition~\ref{overNUM2}, we get $\dim_{H}(\mu_F)>3-t,$
		where $t$ is the affinity dimension of the IFS $\mathcal{I}$ for every $s\in  \bigg[\frac{1+\sqrt{5}}{4},1\bigg)$. Then the claim follows by Theorem~\ref{thm:rapaport}.
	\end{proof}

\subsection{Almost every-type result}
Next, we discuss almost sure results for the above-constructed fractal surfaces. \\
\subsection*{Construction of Graph directed IFS associated with the projection of the Furstenberg IFS} One can see that the group $( \mathcal{G})$ generated by the linear part of the Furstenberg IFS $\mathcal{J}=\{h_1,h_2,h_3,h_4\}$ is a finite group of order $12.$ Precisely, the group elements are as follows:
$$Q_1=\begin{bmatrix} \frac{1}{2} & \frac{\sqrt 3}{2} \\ \frac{\sqrt 3}{2} & \frac{-1}{2} \end{bmatrix}, Q_2=\begin{bmatrix} \frac{1}{2} & \frac{-\sqrt 3}{2} \\ \frac{-\sqrt 3}{2} & \frac{-1}{2} \end{bmatrix}, Q_3=\begin{bmatrix} -1 & 0 \\ 0 & 1 \end{bmatrix}, Q_4=\begin{bmatrix} -1 & 0 \\ 0 & -1 \end{bmatrix},$$
$$Q_5=\begin{bmatrix} 1 & 0 \\ 0 & 1 \end{bmatrix}, Q_6=\begin{bmatrix} \frac{-1}{2} & \frac{-\sqrt 3}{2} \\ \frac{-\sqrt 3}{2} & \frac{1}{2} \end{bmatrix}, Q_7=\begin{bmatrix} \frac{-1}{2} & \frac{\sqrt 3}{2} \\ \frac{\sqrt 3}{2} & \frac{1}{2} \end{bmatrix}, Q_8=\begin{bmatrix} \frac{-1}{2} & \frac{\sqrt 3}{2} \\ \frac{-\sqrt 3}{2} & \frac{-1}{2} \end{bmatrix},$$	
$$Q_9=\begin{bmatrix} \frac{-1}{2} & \frac{-\sqrt 3}{2} \\ \frac{\sqrt 3}{2} & \frac{-1}{2} \end{bmatrix}, Q_{10}=\begin{bmatrix} 1 & 0 \\ 0 & -1 \end{bmatrix}, Q_{11}=\begin{bmatrix} \frac{1}{2} & \frac{-\sqrt 3}{2} \\ \frac{\sqrt 3}{2} & \frac{1}{2} \end{bmatrix}, Q_{12}=\begin{bmatrix} \frac{1}{2} & \frac{\sqrt 3}{2} \\ \frac{-\sqrt 3}{2} & \frac{1}{2} \end{bmatrix}.$$	
First, we consider a direction $v=(1,1)\in \mathbb{R}^2$. Now, we construct a graph directed IFS associated with the Furstenberg IFS $\mathcal{J}$. First, we define a set of vertices $\mathcal{V}:=\{v_1,v_2,\dots,v_{12}\}$ such that $v_i=Q_{i}^{T}v.$ The set $\mathcal{E}_{l,m}$ denotes the set of all directed edges from vertex $v_{l}$ to $v_{m}$. For $l,m\in \mathcal{V}$, if these exists a $k\in \{1,2,3,4\}$ such that $v_{l}=Q_{k}v_{m}$, then we define a directed edge $e\in \mathcal{E}_{l,m}$. Since $\#{\{Q_{l}Q_{k}:k\in {1,2,3,4}\}}=4$ for each $l\in \{1,2,\dots,12\}$, there are only $4$ directed edges from the vertex $v_{l}$ and only $4$ directed edges toward the vertex $v_{l}$.  The set of directed edges is defined by $\mathcal{E}:=\{\mathcal{E}_{l,m}:1\leq l,m\leq 12\}$. The directed graph is denoted by $G(\mathcal{V},\mathcal{E}).$ For $e\in \mathcal{E}_{l,m}$, the associated map $f_{e}$ is defined by 
$$f_{e}(x)=\frac{1}{2s}x+v_{l}\cdot t_{k},$$
where $k\in \{1,2,3,4\}$ such that $v_{l}=Q_{k}v_{m}$ and $t_k$ is the translation vector corresponding to the map $h_{k}$ in the Furstenberg IFS $\mathcal{J}$. Let $A_F$ be the attractor of the IFS $\mathcal{J}$. One can also see that for each $l\in \{1,2,\dots,12\}$, we have $v_{l}\cdot t_{k_1}\ne v_{l}\cdot t_{k_2}$ for all $k_{1}\ne k_{2}\in \{1,2,3,4\}$. The notation $\text{Proj}_{v}(A)$ denotes the projection of the set $A$ in the direction $v$. Thus, we have 
\begin{align*}
   \text{Proj}_{v_l}(A_F)&= \text{Proj}_{v_l}\bigg(\bigcup_{k=1}^{4}h_k(A_F)\bigg)=\bigcup_{k=1}^{4}\text{Proj}_{v_l}\bigg(\frac{1}{2s}Q_{k}(A_F)+t_{k}\bigg)\\&= \bigcup_{k=1}^{4}\bigg(\frac{1}{2s}\text{Proj}_{Q_{k}^{T}v_l}(A_F)+\text{Proj}_{v_l}(t_{k})\bigg)=\bigcup_{m=1}^{12}\bigcup_{e\in \mathcal{E}_{l,m}}f_{e}( \text{Proj}_{v_m}(A_F)).
\end{align*}
Thus, $\bigcup_{l=1}^{12}\text{Proj}_{v_l}(A_F)$ is the attractor of the graph directed IFS $\{f_{e}: e\in \mathcal{E}\}$ with directed graph $G(\mathcal{V},\mathcal{E}).$\\

\subsection*{Induced Markov chain} Let $X_n$ be a Markov chain on the group $\mathcal{G}$. For each $l\in \{1,2,\dots,12\}$, one can see that  $\#{\{Q_{l}Q_{k}:k\in {1,2,3,4}\}}=4$. Thus, we define the transition probability for the Markov chain by 
$$\mathbb{P}(X_{n+1}=Q_m | X_n=Q_l)=\begin{cases}
    \frac{1}{4}\quad  \text{if} \quad \exists~ k\in \{1,2,3,4\}~ \text{with}~ Q_{l}Q_{k}=Q_{m}\\ 0 \quad \text{otherwise}.
\end{cases}$$
Let $P$ be the transition matrix associated with the Markov chain $X_n$. Thus, the matrix $P$ is as follows
$$[P]_{i,j}=\begin{cases}
    \frac{1}{4}\quad  \text{if} \quad \exists~ k\in \{1,2,3,4\}~ \text{with}~ Q_{i}Q_{k}=Q_{j}\\ 0 \quad \text{otherwise},
\end{cases}$$
where $[P]_{i,j}$ is the $ij$th entry of the matrix $P$ for $1\leq i,j\leq 12$. 

We have 
$$Q_2 \cdot Q_1=Q_8,~ Q_1 \cdot Q_2=Q_9,~ Q_2 \cdot Q_4=Q_7,~ Q_1 \cdot Q_4=Q_6=Q_4 \cdot Q_1,~ $$
$$Q_1 \cdot Q_3=Q_8,~ Q_2 \cdot Q_3=Q_9,~Q_{11} \cdot Q_1=Q_7, Q_{12} \cdot Q_{1}=Q_{10},$$
$$Q_1 \cdot Q_1= Q_2 \cdot Q_2=Q_3 \cdot Q_3=Q_4 \cdot Q_4=Q_5=Id,~Q_4 \cdot Q_3=Q_{10},~ Q_1 \cdot Q_{10}=Q_{11},~ Q_2 \cdot Q_{10}=Q_{12}.$$
Now, we define two sets $A=\{Q_5,Q_6,Q_7,Q_8,Q_9,Q_{10}\}$ and $B=\{Q_1,Q_2,Q_3,Q_4,Q_{11},Q_{12}\}$. By the above one can see that for each $Q_i\in A$, there exist a $k\in \{1,2,3,4\}$ such that $Q_iQ_{k}\in B$ and vice versa. Thus, the directed graph associated with the transition matrix $P$ is bipartite graph. So, by considering the arrangement $\{Q_5,Q_6,Q_7,Q_8,Q_9,Q_{10},Q_1,Q_2,Q_3,Q_4,\\Q_{11},Q_{12}\}$, the transition matrix $P$ can be rewritten in the following form 
\begin{equation*}
P=\begin{bmatrix}
0 & 0 & 0 & 0 & 0 & 0 & 1/4 & 1/4 & 1/4 & 1/4 & 0 & 0\\
0 & 0 & 0 & 0 & 0 & 0 & 1/4 & 0 & 0 & 1/4 & 1/4 & 1/4\\
0 & 0 & 0 & 0 & 0 & 0 & 0 & 1/4 & 0 & 1/4 & 1/4 & 1/4\\
0 & 0 & 0 & 0 & 0 & 0 & 1/4 & 1/4 & 1/4 & 0 & 1/4 & 0\\
0 & 0 & 0 & 0 & 0 & 0 & 1/4 & 1/4 & 1/4 & 0 & 0 & 1/4\\
0 & 0 & 0 & 0 & 0 & 0 & 0 & 0 & 1/4 & 1/4 & 1/4 & 1/4\\
1/4 & 1/4 & 0 & 1/4 & 1/4 & 0 & 0 & 0 & 0 & 0 & 0 & 0\\
1/4 & 0 & 1/4 & 1/4 & 1/4 & 0 & 0 & 0 & 0 & 0 & 0 & 0\\
1/4 & 0 & 0 & 1/4 & 1/4 & 1/4 & 0 & 0 & 0 & 0 & 0 & 0\\
1/4 & 1/4 & 1/4 & 0 & 0 & 1/4 & 0 & 0 & 0 & 0 & 0 & 0\\
0 & 1/4 & 1/4 & 1/4 & 0 & 1/4 & 0 & 0 & 0 & 0 & 0 & 0\\
0 & 1/4 & 1/4 & 0 & 1/4 & 1/4 & 0 & 0 & 0 & 0 & 0 & 0\\
\end{bmatrix}
\end{equation*}
Since $Q_{5}$ is identity matrix ($Id$), we have 
$$\mathbb{P}(X_{n}= Id | X_0=Id)=\begin{bmatrix}1&0&\cdots&0 \end{bmatrix}P^{n}\begin{bmatrix}1\\0\\\vdots\\0 \end{bmatrix}=\frac{\#\{\tau\in \{1,2,3,4\}^n: Q_{\tau}=Id\}}{4^n}.$$ 
Clearly, $\mathbb{P}(X_{2n+1}= Id | X_0=Id)=0$ for $n\in \mathbb{N}\cup \{0\}$ and $\mathbb{P}(X_{2}= Id | X_0=Id)=\frac{1}{4}.$
The period of the directed graph associated with the transition matrix $P$ is defined as follows
$$\text{period}=\text{l.c.d.}\{n\in \mathbb{N}: \mathbb{P}(X_{n}= Id | X_0=Id)>0\}.$$ This implies that $\text{period}$ is $2$. But, $P^2=\begin{bmatrix} R&0\\0&S\end{bmatrix}$, where $R$ and $S$ are $6
\times 6$ matrices. The matrix $R$ is as follows
\begin{equation*}
R=\begin{bmatrix}
1/4 & 1/8 & 1/8 & 3/16 & 3/16 & 1/8\\
1/8 & 1/4 & 3/16 & 1/8 & 1/8 & 3/16\\
1/8 & 3/16 & 1/4 & 1/8 & 1/8 & 3/16\\
3/16 & 1/8 & 1/8 & 1/4 & 3/16 & 1/8\\
3/16& 1/8 & 1/8 & 3/16 & 1/4 & 1/8\\
1/8 & 3/16 & 3/16 & 1/8 & 1/8 & 1/4\\
\end{bmatrix}
\end{equation*}
The above matrix $R$ is irreducible and aperiodic (period is 1). Thus by Perron-Frobenius theorem, we get 
$$R^n\to \begin{bmatrix}1\\1\\1\\1\\1\\1 \end{bmatrix} \begin{bmatrix}1/6&1/6&1/6&1/6&1/6&1/6\end{bmatrix}\quad \text{as}\quad  n\to \infty.$$ This implies that 
$$\mathbb{P}(X_{2n}= Id | X_0=Id)=\frac{\#\{\tau\in \{1,2,3,4\}^{2n}: Q_{\tau}=Id\}}{4^{2n}}\to \frac{1}{6}\quad \text{as}\quad n\to \infty.$$
Thus, these exists an $N_{0}\in \mathbb{N}$ such that 
\begin{equation}\label{eq:boundid}
    \#\{\tau\in \{1,2,3,4\}^{2n}: Q_{\tau}=Id\}\geq  \frac{4^{2n}}{12} \quad \forall~~n\geq N_{0}.
\end{equation}
{
\subsection*{Non-degeneracy of constructed GD-IFS} Now, we will show that the constructed graph directed IFS is non-degenerate, which will be useful for determining almost sure results. The set of all infinite-length edges is denoted by ${\mathcal{E}}^*$. Let $\bold{e}=(e_1,e_2,\dots), \bold{e'}=(e'_1,e'_2,\dots)\in {\mathcal{E}}^*$ such that $\bold{e}\ne \bold{e'}$ and $e_1\in \mathcal{E}_{l,m_1}, e'_1\in \mathcal{E}_{l,m_2}.$ Let $m\in \mathbb{N}$ be smallest  number such that $e_m\ne e'_{m}$. Then, there exist $k_1\ne k_2\in \{1,2,\dots,12\}$ such that $e_m\in \mathcal{E}_{l_1,m_{k_1}}$ and $e'_m\in \mathcal{E}_{l_1,m_{k_2}}$. Let $\Pi$ be the natural projection corresponding to the constructed graph-directed IFS. Thus, we have 
 $$\Pi(\sigma^{m-1} \bold{e})=f_{e_m}(\Pi(\sigma^{m} \bold{e}))=\frac{1}{2s}(\Pi(\sigma^{m} \bold{e}))+v_{l_1}\cdot t_{r_1}$$
$$\Pi(\sigma^{m-1} \bold{e'})=f_{e'_m}(\Pi(\sigma^{m} \bold{e'}))=\frac{1}{2s}(\Pi(\sigma^{m} \bold{e'}))+v_{l_1}\cdot t_{r_2},$$ where $r_1\ne r_2$. Since $v_{l_1}\cdot t_{r_1}\ne v_{l_1}\cdot t_{r_2}$, we have $\Pi(\sigma^{m-1} \bold{e})\not\equiv \Pi(\sigma^{m-1} \bold{e'})$ as analytic functions of $s$ on $(1/2,\infty)$, and in particular, on $(1/2,1/)$. Thus, we get $\Pi(\bold{e})\not\equiv \Pi(\bold{e'})$. This implies that the graph-directed IFS is non-degenerate. \par
}

Now, by using the above constructed GD-IFS, we will prove typical type results for the Hausdorff dimension.
\begin{proof}[Proof of Theorem~\ref{GH2}] The upper bound follows by \eqref{eq:GHub}. Now, we show the lower bound. \par 
 
Set $\mathcal{N}_{n}:=\#\{\tau\in \{1,2,3,4\}^{2n}: Q_{\tau}=Id\}.$
For $n\geq N_0$, we define an self-affine IFS $$\Phi_{n}:=\{W_{\tau}: Q_{\tau}=Id~\text{and}~ \tau\in \{1,2,3,4\}^{2n}\}.$$ Let $A_{\Phi_n}$ be the attractor of the IFS $\Phi_n$.
The affinity dimension ($t_n$) of the self-affine IFS $\Phi_n$ is uniquely determined by following equation
$$\sum_{\substack{\tau\in \{1,2,3,4\}^{2n}\\Q_{\tau}=Id}} s^{2n}\bigg(\frac{1}{2^{2n}}\bigg)^{t_{n}-1}=1$$
Let $(p_{i})_{i=1}^{\mathcal{N}_n}$ be a probability vector such that $p_i=\frac{1}{\mathcal{N}_n}$ for all $1\leq i\leq \mathcal{N}_n$. Let $\mathcal{F}_{n}:=\{h_{\tau}: Q_{\tau}=Id~\text{and}~ \tau\in \{1,2,3,4\}^{2n}\}\}$ be the Furstenberg IFS corresponding to the IFS $\Phi_{n}$. Let $\mu_{\mathcal{F}}^{n}$ be the invariant measure corresponding to the IFS $\mathcal{F}_{n}$ with probability vector $(p_{i})_{i=1}^{\mathcal{N}_n}.$ Since the graph directed IFS $\{f_{e}: e\in \mathcal{E}\}$ with directed graph $G(\mathcal{V},\mathcal{E})$ corresponding to the Furstenberg IFS $\mathcal{J}$ is non-degenerate, the projection of the $\mathcal{F}_{n}$ on the direction $v=(1,1)\in \mathbb{R}^2$ is also non-degenerate. Then, by the result of Hochman \cite{Hochman2014}, there exist a set $\mathcal{E}_{n}\subset (\frac{1}{2},1)$ with $\dim_{H}(\mathcal{E}_{n})=0$ such that 
$$\dim_{H}(\mu_{\mathcal{F}}^{n})\geq \min\bigg\{1,\frac{-\sum_{i=1}^{\mathcal{N}_n}p_i\log p_i}{\sum_{i=1}^{\mathcal{N}_n}p_i\log (2s)^{2n}}\bigg\}=\min\bigg\{1,\frac{\log \mathcal{N}_n}{2n\log (2s)}\bigg\}\quad \forall ~s\in \bigg(\frac{1}{2},1\bigg)\setminus \mathcal{E}_{n}.$$
Now, by using the estimate \eqref{eq:boundid} of $\mathcal{N}_n$, we get 
$$\dim_{H}(\mu_{\mathcal{F}}^{n})\geq \min\bigg\{1,\frac{\log 4}{\log (2s)}-\frac{\log 12}{2n \log 2s}\bigg\}\quad \forall ~s\in \bigg(\frac{1}{2},1\bigg)\setminus \mathcal{E}_{n}.$$
Let $\hat{t}_n \in \mathbb{R}$ be such that 
$$\sum_{\substack{\tau\in \{1,2,3,4\}^{2n}\\Q_{\tau}=Id}} s^{2n}\bigg(\frac{1}{2^{2n}}\bigg)^{\hat{t}_{n}-1}\geq \frac{4^{2n}}{12}s^{2n}\bigg(\frac{1}{2^{2n}}\bigg)^{\hat{t}_{n}-1}=1.$$
This implies that $t_{n}\geq \hat{t}_n$ and 
$$\hat{t}_{n}=1+\frac{\log 4s}{\log 2}-\frac{\log 12}{2n \log 2}=3+\frac{\log s}{\log 2}-\frac{\log 12}{2n \log 2}.$$
Clearly, $\hat{t}_n \to t$ as $n\to \infty$. Since $\frac{\log 4}{\log (2s)}>-\frac{\log s}{\log 2}~ \forall~ s\in \big(\frac{1}{2},1\big)$, we have 
$$\min\bigg\{1,\frac{\log 4}{\log (2s)}-\frac{\log 12}{2n \log 2s}\bigg\}>3-\hat{t}_{n}\geq 3-t_n~ \forall~ s\in \bigg(\frac{1}{2},1\bigg)$$
for all $n\geq N_1$, where $N_1\in \mathbb{N}$ is some large number. This implies that for $n\geq \max\{N_0,N_1\}$ we get
$$\dim_{H}(\mu_{\mathcal{F}}^{n})>3-t_n\quad \forall ~s\in \bigg(\frac{1}{2},1\bigg)\setminus \mathcal{E}_{n}.$$
Then, by Theorem~\ref{thm:rapaport}, for $n\geq \max\{N_0,N_1\}$ we obtain
 $$\dim_{H}(A_{\Phi_n})=t_n\geq \hat{t}_n \quad \forall ~s\in \bigg(\frac{1}{2},1\bigg)\setminus \mathcal{E}_{n}.$$
Set $\mathcal{E}:=\bigcup_{n=\max\{N_0,N_1\}} \mathcal{E}_n.$ Thus, $\dim_{H}(\mathcal{E})=0.$ This implies that 
$$\dim_{H}(G(f^*))\geq \dim_{H}(A_{\Phi_n})\geq \hat{t}_n \quad \forall ~s\in \bigg(\frac{1}{2},1\bigg)\setminus \mathcal{E}.$$ Thus, for all $s\in \big(\frac{1}{2},1\big)\setminus \mathcal{E},$ we get
$$\dim_{H}(G(f^*))\geq t \quad \forall ~s\in \bigg(\frac{1}{2},1\bigg)\setminus \mathcal{E}.$$
Since $\dim_{H}(G(f^*))\leq \overline{\dim}_{B}(G(f^*)) \leq t=3+\frac{\log(s)}{\log(2)},$ we have 
$$\dim_{H}(G(f^*))={\dim}_{B}(G(f^*))= t=3+\frac{\log(s)}{\log(2)}$$
for all $s\in \big(\frac{1}{2},1\big )\setminus \mathcal{E}$. This completes the proof.
\end{proof}

\bibliographystyle{amsplain}


\end{document}